\documentclass[review,12pt]{elsarticle}
\usepackage{amsmath}
\usepackage{amsfonts}
\usepackage{geometry}   
\usepackage[parfill]{parskip}    
\usepackage{graphicx}
\usepackage{amssymb}
\usepackage{epstopdf}
\usepackage{psfrag}

\newcommand {\bea}{\begin{eqnarray}}
\newcommand {\ea}{\end{eqnarray}}
\newtheorem{proposition}{Proposition}[section]
\newtheorem{theorem}{Theorem}[section]
\newtheorem{Assumption}{Assumption}[section]
\newtheorem{lemma}{Lemma}[section]

\newenvironment{proof}[1][Proof]{\textbf{#1.} }{\hspace{\stretch{1}}\rule{0.5em}{0.5em}}

\usepackage{xspace}

\usepackage{subfigure}

\newcommand{\thmref}[1]{{Theorem~\ref{#1}}}
\newcommand{\lemref}[1]{{Lemma~\ref{#1}}}

\newcommand{\assref}[1]{{Assumption~\ref{#1}}}
\newcommand{\propref}[1]{{Proposition~\ref{#1}}}

\newcommand{\rmref}[1]{{Remark~\ref{#1}}}
\begin{document}
\begin{frontmatter}
\title{Strong convergence  of the backward Euler approximation for the finite element discretization of semilinear parabolic SPDEs with non-global Lipschitz drift driven by additive noise}

\author[jdm]{Jean Daniel Mukam}
\ead{jean.d.mukam@aims-senegal.org}
\address[jdm]{Fakult\"{a}t f\"{u}r Mathematik, Technische Universit\"{a}t Chemnitz, 09126 Chemnitz, Germany.}

\author[at,atb,atc]{Antoine Tambue}
\cortext[cor1]{Corresponding author}
\ead{antonio@aims.ac.za}
\address[at]{Department of Computing Mathematics and Physics,  Western Norway University of Applied Sciences, Inndalsveien 28, 5063 Bergen.}
\address[atb]{Center for Research in Computational and Applied Mechanics (CERECAM), and Department of Mathematics and Applied Mathematics, University of Cape Town, 7701 Rondebosch, South Africa.}
\address[atc]{The African Institute for Mathematical Sciences(AIMS) of South Africa,
6-8 Melrose Road, Muizenberg 7945, South Africa.}


\begin{abstract}
This paper deals with the backward Euler method applied to  semilinear parabolic stochastic partial differential equations (SPDEs) driven by additive noise. 
The SPDE is discretized in space by the finite element method and in time by the backward Euler. 
 We consider a larger  class  of nonlinear drift functions, which are of  Nemytskii type and polynomial  of any odd degree with negative leading term,  instead of only dealing  with the special case of stochastic Allen-Chan equation as in the up to date literature.
Moreover our linear operator is of second order and  not necessarily self-adjoint, therefore makes   estimates more challenging than in the case of self-adjoint operator. We prove the strong convergence of our fully discrete schemes toward the mild solution and  results indicate how the convergence rates depend on the regularities of the initial data and the noise. In particular, for trace class noise, we achieve convergence order $\mathcal{O} (h^{2}+\Delta t^{1-\epsilon})$, where $\epsilon>0$ is positive number, small enough.  We also provide numerical experiments to illustrate our theoretical results.
\end{abstract}

\end{frontmatter}
\section{Introduction}
\label{intro}
We consider numerical approximations of SPDE defined in $\Lambda\subset \mathbb{R}^d$, $d=1,2, 3$,  with initial value and  boundary conditions of the following type
\begin{eqnarray}
\label{model}
dX(t)+AX(t)dt=F(X(t))dt+dW(t), \quad X(0)=X_0, \quad t\in(0,T]
\end{eqnarray}
 on the Hilbert space $L^2(\Lambda)$, where    $T>0$ is the final time,   
$A$ is an unbounded  linear operator, not necessarily self-adjoint.
The noise  $W(t)=W(x,t)$ is a $Q-$Wiener process defined in a filtered probability space $(\Omega,\mathcal{F}, \mathbb{P}, \{\mathcal{F}_t\}_{t\geq 0})$ and  $Q : H\longrightarrow H$ is a positive linear self-adjoint operator.
The filtration is assumed to fulfill the usual conditions (see e.g., \cite[Definition 2.1.11]{Prevot}). Note that the noise can be represented as follows, see e.g., \cite{Prato,Prevot}
\begin{eqnarray}
\label{noise}
W(x,t)=\sum_{i=0}^{\infty}\sqrt{q_i}e_i(x)\beta_i(t), \quad t\in[0,T],\quad x\in H,
\end{eqnarray} 
where $q_i$, $e_i$, $i\in\mathbb{N}^d$ are respectively the eigenvalues and the eigenfunctions of the covariance operator $Q$,
and $\beta_i$ are independent and identically distributed standard Brownian motions.
 Precise assumptions on $F$, $B$,  $X_0$ and $A$  will be given in the next section to ensure the existence of the unique mild solution $X$ of \eqref{model}, 
 which has the following representation, see e.g.,  \cite{Prato,Prevot}
\begin{eqnarray}
\label{mild1}
X(t)=S(t)X_0+\int_0^tS(t-s)F(X(s))ds+\int_0^tS(t-s)dW(s), \quad t\in[0,T].
\end{eqnarray}
Equations of type \eqref{model} are used to model different real world phenomena in different fields such as biology, chemistry, physics, see e.g., \cite{Shardlow,SebaGatam,ATthesis}. In many cases, 
explicit solutions of SPDEs are unknown, therefore numerical methods are  the only good tools to provide realistic approximations. Numerical approximation of SPDE  \eqref{model}  
is therefore an active research area and have attracted a lot of attentions since two decades, see e.g., \cite{Antonio1,Xiaojie1,Xiaojie2,Raphael,Jentzen1,Jentzen2,Yan1}  and references therein.
The convergence analysis of many numerical schemes in the literature are based on the global Lipschitz condition of the drift function $F$. However,  for many physical problems, the nonlinear function $F$ does not satisfy the global Lipschitz condition. For instance for the stochastic Allen-Cahn equation, the nonlinear drift function is of type $F(u)=u-u^3$, which is not globally Lipschitz, see e.g., \cite{Brehier1,Kovac1}. It is well known that  the standard Euler-Maruyama method for stochastic differential equations (SDEs) with non-global Lipschitz drift diverges \cite{Hutzenthaler1}.  For SDEs with non-global Lipschitz condition, implicit schemes were proved to  converge strongly to the exact solution \cite{Xuerong}. Recently,  explicit taming methods were proved to be efficient for such SDEs, see e.g., \cite{Hutzenthaler2,AntjMaster}. This taming strategy is being extended to the case of SPDEs with non-global Lipschitz drift function, see e.g., \cite{Becker1,Xiaojie4}. In fact, numerical approximations of SPDEs with non-global Lipschitz drift is  currently a hot topic in numerical analysis. Campbell and Lord \cite{Lord1} proposed an adaptive time-stepping scheme for such problems.  Br\'{e}hier et al. \cite{Brehier1};   Br\'{e}hier and Gouden\`{e}ge \cite{Brehier2} analyze the converge of some explicit splitting approximations for  stochastic Allen-Chan equation. All the above mentioned schemes are explicit and are of Euler-type. Since backward Euler method is more stable than the explicit Euler schemes,
   they have been considered in the literature.
    Kov\'{a}cs et al. \cite{Fredrik2,Fredrik1} analyze the backward Euler and a fully implicit split-step schemes for stochastic Allen-Cahn equation. However, to the best of our knowledge in the above mentioned works,  the linear operator is assumed to be the Laplace operator, which is self-adjoint and therefore makes the  estimations more easier than they would be in more general non-self adjoint case.
    
      In this paper, we consider  the full discretization of a general semilinear second order SPDE  where  the linear operator  is non-necessarily self-adjoint and therefore makes estimates more challenging. 
       We discretise in space  the SPDE  by the finite element method and in time by the backward Euler  method. Note that in many of the above cited references, the convergence results are done only for the particular case of stochastic Allen-Cahn equation. In fact, one key argument in such situation is the Sobolev embedding $H^1_0(\Lambda)\subset L^6(\Lambda)$, which is due to the fact the drift function is polynomial of degree $3$ and the spatial dimension is $d\leq 3$,  see e.g. \cite{Fredrik1,Fredrik2}. 
       For  more general cases, the strong convergence rates of numerical schemes are far from being understood and it was still an open problem to derive general
strong convergence rates of numerical schemes for SPDEs with non-global Lipschitz coefficients.
 Recently Campbell and Lord \cite{Lord1} proposed  an adaptive time-stepping method  to solve such problems. However, the assumptions  on the nonlinear drift function made in \cite{Lord1}  do  not include the case of stochastic Allen-Cahn equation investigated in \cite{Fredrik1,Fredrik2,Brehier2,Brehier1},  therefore  these assumptions are still restrictive. In fact, as the solution process is sough in the Hilbert $L^2(\Lambda)$,  the function $F(u)=u-u^3$, $u\in L^2(\Lambda)$ does not satisfy the polynomial growth condition stated in \cite{Lord1} (see e.g. \cite[Assumption 2.4]{Lord1}). 
 In this work, we derive  general strong convergence  rates for SPDEs with nonlinear drift of Nemytskii type, which include the case of stochastic Allen-Cahn equation.
  More, precisely, we consider $F(u)(x)=\varphi(u(x))$, $u\in L^2(\Lambda)$ and $x\in\Lambda$, where $\varphi :\mathbb{R}\longrightarrow\mathbb{R}$ is polynomial of any odd order with negative leading coefficient. Hence our setting is more general than the ones in the current scientific literature  and generalize  the case of stochastic  Allen-Chan equation. 
   In almost all the above mentioned references, the optimal convergence order is $\frac{1}{2}$. For instance the optimal convergence order in \cite{Brehier2,Brehier1} is $\frac{1}{2}$,  providing that $\Vert A^{\frac{1}{2}}Q^{\frac{1}{2}}\Vert_{\mathcal{L}_2(H)}< \infty$ (which corresponds to \assref{assumption4} with $\beta=2$). In this paper, we prove that our numerical schemes achieve rather higher convergence orders of the form $\mathcal{O}\left(h^{\beta}+\Delta t^{\frac{\beta}{2}-\epsilon}\right)$, where $\epsilon>0$ is an arbitrarily number small enough. So for $\beta=2$, we achieve an optimal convergence  rate  $1-\epsilon$ in time.

The rest of this paper is organized as follows. Section \ref{wellposed} deals with the well-posedness problem,  the numerical scheme together with the main result. 
In Section \ref{Convergence}, we provide the proof of the main result. We end the paper in Section \ref{numerik} with a numerical experiment illustrating our theoretical result.

 \section{Mathematical setting and numerical method}
 \label{wellposed}
\subsection{Main assumptions and well-defined problem}
\label{notation}
Let us define functional spaces, norms and notations that  will be used in the rest of the paper. Let $\mathcal{C}:=\mathcal{C}(\overline{\Lambda}, \mathbb{R})$ be the set of continuous functions equipped with the norm $\Vert u\Vert_{\mathcal{C}}=\sup\limits_{x\in\overline{\Lambda}}\vert u(x)\vert$, $u\in\mathcal{C}$ Let  $(H,\langle.,.\rangle,\Vert .\Vert)$ be a separable Hilbert space and $(U, \Vert .\Vert_U)$ a Banach space. We denote by $L^p(\Omega, U)$, $p\geq 2$ the Banach space of all equivalence classes of $p$-integrable $U$-valued random variables. The norms in the Sobolev spaces $H^m(\Lambda),\, m \geqslant 0$ will be denoted by
$\Vert. \Vert_m$. By $L(U,H)$, we denote
 the space of bounded linear mappings from $U$ to $H$ endowed with the usual  operator norm $\Vert .\Vert_{L(U,H)}$, by  $\mathcal{L}_2(U,H):=HS(U,H)$ we denote
  the space of Hilbert-Schmidt operators from $U$ to $H$. 
  We equip $\mathcal{L}_2(U,H)$ with the norm
 \begin{eqnarray}
 \label{def1}
 \Vert l\Vert^2_{\mathcal{L}_2(U,H)} :=\sum_{i=1}^{\infty}\Vert l\psi_i\Vert^2,\quad l\in \mathcal{L}_2(U,H),
 \end{eqnarray}
 where $(\psi_i)_{i=1}^{\infty}$ is an orthonormal basis of $U$. Note that \eqref{def1} is independent of the orthonormal basis of $U$.
  For simplicity, we use the notations $L(U,U)=:L(U)$,  $\mathcal{L}_2(U,U)=:\mathcal{L}_2(U)$ and in the case $U=Q^{\frac{1}{2}}(H)$ we use the notation $L^0_2:=\mathcal{L}_2\left(Q^{\frac{1}{2}}(H), H\right)$.
  It is well known that for all $l\in L(U,H)$ and $l_1\in\mathcal{L}_2(U)$, it holds that $ll_1\in\mathcal{L}_2(U,H)$ and 
  \begin{eqnarray}
  \label{trace1}
  \Vert ll_1\Vert_{\mathcal{L}_2(U,H)}\leq \Vert l\Vert_{L(U,H)}\Vert l_1\Vert_{\mathcal{L}_2(U)}.
  \end{eqnarray}
Let $p\in[2,\infty)$ and $t\in[0, T]$.
 For any $L^0_2$-valued predictable process $\phi(s)$, $0\leq\tau_1\leq s\leq \tau_2\leq T$, the following estimate called Burkh\"{o}lder-Davis-Gundy inequality holds
\begin{eqnarray}
\label{Davis1}
\left\Vert\int_{\tau_1}^{\tau_2}\phi(s)dW(s)\right\Vert_{L^p(\Omega, H)}&\leq& C(p)\left(\int_{\tau_1}^{\tau_2}\Vert \phi(s)\Vert^2_{L^p\left(\Omega,L^0_2\right)}ds\right)^{1/2}\nonumber\\
&=& C(p)\left(\int_{\tau_1}^{\tau_2}\Vert \phi(s)Q^{\frac{1}{2}}\Vert^2_{L^p\left(\Omega,\mathcal{L}_2\right)}ds\right)^{1/2},
\end{eqnarray}
see e.g., \cite[Proposition 2.6]{Raphael1} or \cite[Theorem 4.36]{Prato}.
 
 Throughout this paper, we assume that $\Lambda$ is bounded and has smooth boundary or is a convex polygon of $\mathbb{R}^d$, $d=1,2,3$. From now we consider  the linear operator $A$ to be given by
\begin{eqnarray}
\label{operator}
Au=-\sum_{i,j=1}^{d}\dfrac{\partial}{\partial x_i}\left(D_{ij}(x)\dfrac{\partial u}{\partial x_j}\right)+\sum_{i=1}^dq_i(x)\dfrac{\partial u}{\partial x_i},\quad
\mathbf{D}=\left(D_{i,j} \right)_{1\leq i,j \leq d},\,\,\,\,\,\,\, \mathbf{q}=\left( q_i \right)_{1 \leq i \leq d},
\end{eqnarray}
where $D_{ij}\in L^{\infty}(\Lambda)$, $q_i\in L^{\infty}(\Lambda)$. We assume that there exists a  constant $c_1>0$ such that 
\begin{eqnarray}
\label{ellipticity}
\sum_{i,j=1}^dD_{ij}(x)\xi_i\xi_j\geq c_1|\xi|^2, \quad \xi\in \mathbb{R}^d,\quad x\in\overline{\Lambda}.
\end{eqnarray}
As in \cite{Antonio1,Suzuki} we introduce two spaces $\mathbb{H}$ and $V$, such that $\mathbb{H}\subset V$; the two spaces depend on the boundary 
conditions and the domain of the operator $A$. For  Dirichlet (or first-type) boundary conditions we take 
\begin{eqnarray*}
V=\mathbb{H}=H^1_0(\Lambda)=\overline{C^{\infty}_{c}(\Lambda)}^{H^1(\Lambda)}=\{v\in H^1(\Lambda) : v=0\; \text{on}\; \partial \Lambda\}.
\end{eqnarray*}
For Robin (third-type) boundary condition and  Neumann (second-type) boundary condition, which is a special case of Robin boundary condition, we take $V=H^1(\Lambda)$
\begin{eqnarray*}
\mathbb{H}=\{v\in H^2(\Lambda) : \partial v/\partial \mathtt{v}_{ A}+\alpha_0v=0,\; \text{on}\; \partial \Lambda\}, \quad \alpha_0\in\mathbb{R},
\end{eqnarray*}
where $\partial v/\partial \mathtt{v}_{ A}$ is the normal derivative of $v$ and $\mathtt{v}_{ A}$ is the exterior pointing normal $n=(n_i)$ to the boundary of $A$, given by
\begin{eqnarray*}
\partial v/\partial\mathtt{v}_{A}=\sum_{i,j=1}^dn_i(x)D_{ij}(x)\dfrac{\partial v}{\partial x_j},\,\,\; x \in \partial \Lambda.
\end{eqnarray*}
Using the Green's formula and the boundary conditions, we obtain the following corresponding bilinear form associated to $A$  
\begin{eqnarray*}
a(u,v)=\int_{\Lambda}\left(\sum_{i,j=1}^dD_{ij}\dfrac{\partial u}{\partial x_i}\dfrac{\partial v}{\partial x_j}+\sum_{i=1}^dq_i\dfrac{\partial u}{\partial x_i}v\right)dx, \quad u,v\in V,
\end{eqnarray*}
for Dirichlet and Neumann boundary conditions, and  
\begin{eqnarray*}
a(u,v)=\int_{\Lambda}\left(\sum_{i,j=1}^dD_{ij}\dfrac{\partial u}{\partial x_i}\dfrac{\partial v}{\partial x_j}+\sum_{i=1}^dq_i\dfrac{\partial u}{\partial x_i}v\right)dx+\int_{\partial\Lambda}\alpha_0uvdx, \quad u,v\in V,
\end{eqnarray*}
for Robin boundary conditions. Using  G\aa rding's inequality (see e.g., \cite{ATthesis}), it holds that there exist two positive constants $c_0$ and $\lambda_0$ such that
\begin{eqnarray}
\label{ellip1}
a(v,v)\geq \lambda_0\Vert v \Vert^2_{1}-c_0\Vert v\Vert^2, \quad v\in V.
\end{eqnarray}
By adding and substracting $c_0Xdt$ in both sides of \eqref{model}, we obtain a new linear operator
 still denoted by $A$, and the corresponding  bilinear form is also still denoted by $a$. Therefore, the following coercivity property holds
\begin{eqnarray}
\label{ellip2}
a(v,v)\geq \lambda_0\Vert v\Vert^2_1,\quad v\in V.
\end{eqnarray}
Note that the expression of the nonlinear term $F$ has changed as we included the term $c_0X$ in a new nonlinear term that we still denote by  $F$. The coercivity property (\ref{ellip2}) implies that $A$ is sectorial in $L^{2}(\Lambda)$, i.e. there exist $C_{1},\, \theta \in (\frac{1}{2}\pi,\pi)$ such that
\begin{eqnarray*}
 \Vert (\lambda I -A )^{-1} \Vert_{L(L^{2}(\Lambda))} \leq \dfrac{C_{1}}{\vert \lambda \vert },\;\quad \quad
\lambda \in S_{\theta},
\end{eqnarray*}
where $S_{\theta}:=\left\lbrace  \lambda \in \mathbb{C} :  \lambda=\rho e^{i \phi},\; \rho>0,\;0\leq \vert \phi\vert \leq \theta \right\rbrace $ (see \cite{Henry}).
 Then  $-A$ is the infinitesimal generator of a bounded analytic semigroup $S(t)=:e^{-t A}$  on $L^{2}(\Lambda)$  such that
\begin{eqnarray*}
S(t)= e^{-t A}=\dfrac{1}{2 \pi i}\int_{\mathcal{\gamma}_A} e^{ t\lambda}(\lambda I -A)^{-1}d \lambda,\quad t>0,
\end{eqnarray*}
where $\mathcal{\gamma}_A$  denotes a path that surrounds the spectrum of $-A $.
The coercivity  property \eqref{ellip2} also implies that $A$ is a positive operator and its fractional powers are well defined  for any $\alpha>0,$ by
\begin{equation}
\label{fractional}
 \left\{\begin{array}{rcl}
         A^{-\alpha} & =& \frac{1}{\Gamma(\alpha)}\displaystyle\int_0^\infty  t^{\alpha-1}{\rm e}^{-tA}dt,\\
         A^{\alpha} & = & (A^{-\alpha})^{-1},
        \end{array}\right.
\end{equation}
where $\Gamma(\alpha)$ is the Gamma function (see e.g., \cite{Henry}).
Following \cite{Larsson2,Antonio1,Suzuki}, we characterize the domain of the operator $A^{r/2}$ denoted by $\mathcal{D}(A^{r/2})$, $r\in\{1,2\}$ with the following equivalence of norms
\begin{eqnarray*}
\mathcal{D}(A^{r/2})=\mathbb{H}\cap H^{r}(\Lambda), \quad \text{ (for Dirichlet boundary conditions)},\\
\mathcal{D}(A)=\mathbb{H}, \quad \mathcal{D}(A^{1/2})=H^1(\Lambda), \quad \text{(for Robin boundary conditions)}.
\end{eqnarray*}
Endowed with the norm $\Vert A^{r/2}.\Vert$, $\mathcal{D}\left(A^{r/2}\right)$ is a Banach space, see e.g., \cite{Henry}.  The following equivalence of norms holds $\Vert v\Vert_{H^1(\Lambda)}\equiv \Vert A^{r/2}v\Vert=:\Vert v\Vert_r$ for any $v\in\mathcal{D}(A^{r/2})$. For $\gamma\geq 0$, the following Sobolev embedding holds
\begin{eqnarray}
\label{embedd0}
\mathcal{D}\left(A^{\gamma/2}\right)\hookrightarrow \mathcal{C}=C(\overline{\Lambda}, \mathbb{R}),\quad \gamma>\frac{d}{2},\quad d\in\{1,2,3\}.
\end{eqnarray}
Using the embedding \eqref{embedd0} and the fact that the Sobolev constant is independent of the function, the following Sobolev embedding holds
\begin{eqnarray}
\label{embedd0a}
L^{2p}\left(\Omega, \mathcal{D}\left(A^{\gamma/2}\right)\right)\hookrightarrow L^{2p}\left(\Omega, C(\overline{\Lambda}, \mathbb{R})\right),\quad \gamma>\frac{d}{2},\quad p\geq 1,\quad d\in\{1,2,3\}.
\end{eqnarray}

In order to ensure the existence and the uniqueness  solution of \eqref{model} and for the purpose of the convergence analysis, we make the following assumptions.
\begin{Assumption}
\label{assumption1}\textbf{[Initial data]}
The initial data  $X_0\in L^{2p}\left(\Omega, \mathcal{D}\left(A^{\beta/2}\right)\right)\cap L^{2p}\left(\Omega, \mathcal{C}(\overline{\Lambda}, \mathbb{R})\right)$, for $\beta\in\left(\frac{d}{2}, 2\right]$ and  for any $p\in[1, \infty)$.
\end{Assumption}
\begin{Assumption}\textbf{[One-sided Lipschitz condition]}
\label{assumption2}
There exists a positive constant $L_0$ such that the nonlinear function $F: H\longrightarrow H$ satisfies the following estimate
\begin{eqnarray}
\left\langle u-v, F(u)-F(v)\right\rangle \leq L_0\Vert u-v\Vert^2,\quad u,v\in H.
\end{eqnarray} 
 \end{Assumption}
 
 \begin{Assumption}
 \label{assumption3}
 There exists $\varphi: \mathbb{R}\longrightarrow \mathbb{R}$, polynomial of odd degree and negative leading coefficient such that the nonlinear operator $F$ is defined as the Nemytskii operator
 \begin{eqnarray*}
 F\left(u\right)(x)=\varphi\left(u(x)\right),\quad u\in H,\quad x\in \Lambda.
 \end{eqnarray*}
 \end{Assumption}
 
 \begin{Assumption}\textbf{[Covariance operator]}
 \label{assumption4}
 The covariance operator $Q: H\longrightarrow H$ is  satisfies the following estimate
 \begin{eqnarray*}
 \left\Vert A^{\frac{\beta-1}{2}}Q^{\frac{1}{2}}\right\Vert_{\mathcal{L}_2(H)}<\infty,
 \end{eqnarray*}
 where $\beta$ is defined in \assref{assumption1}.
 \end{Assumption}

\begin{lemma}\textbf{[Polynomial growth]}
\label{remarkassumption3}
Under \assref{assumption3}, there exist two constants and  $L_1, c_1\in[0, \infty)$ such that the nonlinear function $F$  satisfies the following 
\begin{eqnarray}
\label{polynome1}
\Vert F(u)\Vert\leq L_1+ L_1\Vert u\Vert\left(1+\Vert u\Vert^{c_1}_{\mathcal{C}}\right), \quad \Vert F(u)\Vert_{\mathcal{C}}\leq L_1+L_1\Vert u\Vert_{\mathcal{C}}^{c_2}\quad u \in H\cap C(\overline{\Lambda}, \mathbb{R}),\\
\label{polynome2}
\quad \Vert F(u)-F(v)\Vert\leq L_1\Vert u-v\Vert\left(1+\Vert u\Vert^{c_1}_{\mathcal{C}}+\Vert v\Vert^{c_1}_{\mathcal{C}}\right),\quad u, v\in H\cap C(\overline{\Lambda},\mathbb{R}).
\end{eqnarray}
 Moreover, $F$ is
 differentiable with polynomial growth derivative, i.e. there exist two constants $L_2, c_2\in[0, \infty)$ such that
\begin{eqnarray}
\label{polynome3}
\Vert F'(u)v\Vert_{L(H)}\leq L_2\Vert v\Vert\left(1+\Vert u\Vert^{c_2}_{\mathcal{C}}\right),\quad u, v\in H\cap \mathcal{C}(\overline{\Lambda}, \mathbb{R}).
\end{eqnarray} 
 \end{lemma}
\begin{proof}
Let us assume without loss of generality that $\varphi$ is polynomial of degree $l>1$, that is
\begin{eqnarray}
\varphi(x)=\sum_{i=0}^la_ix^i,\quad x\in\mathbb{R}.
\end{eqnarray}
Note that the proofs  in the cases $l=0, 1$ are obvious. 
For any $u\in H\cap \mathcal{C}(\overline{\Lambda}, \mathbb{R})$, using traingle inequality and the fact $\left(\sum\limits_{i=0}^lc_i\right)^2\leq (l+1)\sum\limits_{i=0}^lc_i^2$, $c_i\geq 0$, we obtain
\begin{eqnarray*}
\Vert F(u)\Vert^2&=&\int_{\Lambda}\vert F(u)(x)\vert^2dx=\int_{\Lambda}\vert \varphi(u(x))\vert^2\leq (l+1)\sum_{i=0}^l\vert a_i\vert^2\int_{\Lambda}\vert u(x)\vert^{2i}dx\nonumber\\
&\leq&(l+1)\vert a_0\vert^2+(l+1)\vert a_1\vert\int_{\Lambda}\vert u(x)\vert^2dx\nonumber\\
&+&(l+1)\max_{2\leq i\leq l}\sup_{x\in\overline{\Lambda}}\vert u(x)\vert^{2i-2}\sum_{i=2}^l\vert a_i\vert^2\int_{\Lambda}\vert u(x)\vert^2dx\nonumber\\
&\leq&(l+1)\vert a_0\vert^2+(l+1)\vert a_1\vert^2\Vert u\Vert^2+(l+1)\max_{2\leq i\leq 2}\Vert u\Vert^{2i-2}_{\mathcal{C}}\left(\max_{2\leq i\leq l}\vert a_i\vert^2\right)\Vert u\Vert^2\nonumber\\
&\leq&(l+1)\vert a_0\vert^2+(l+1)\vert a_1\vert^2\Vert u\Vert^2+(l+1)\max_{2\leq i\leq 2}\left(\Vert u\Vert_{\mathcal{C}}+1\right)^{2i-2}\left(\max_{2\leq i\leq l}\vert a_i\vert^2\right)\Vert u\Vert^2\nonumber\\
&\leq&(l+1)\vert a_0\vert^2+(l+1)\vert a_1\vert^2\Vert u\Vert^2+(l+1)2^{2l-3}\left(\Vert u\Vert_{\mathcal{C}}^{2l-2}+1\right)\left(\max_{2\leq i\leq l}\vert a_i\vert^2\right)\Vert u\Vert^2\nonumber\\
&\leq& L_1+L_1\Vert u\Vert^2\left(1+\Vert u\Vert^{2l-2}_{\mathcal{C}}\right).
\end{eqnarray*}
This completes the proof of the first estimate of \eqref{polynome1}. The proof of the second estimate of \eqref{polynome1} is similar to the first one.  The proof of \eqref{polynome2} is similar to that of \eqref{polynome1} by using the following well known fact
\begin{eqnarray*}
a^n-b^n=(a-b)\sum_{i=0}^{n-1}a^ib^{n-1-i},\quad a, b\in \mathbb{R},\quad n\geq 1.
\end{eqnarray*}
The proof of \eqref{polynome3} is similar to that of \eqref{polynome1} by noticing that $F'(u)(v)(x)=\varphi'(u(x)).v(x)$, for any $u, v\in H\cap \mathcal{C}(\overline{\Lambda}, \mathbb{R})$ and $x\in\Lambda$. 
\end{proof}

\begin{proposition}
\label{propexistence}
Let Assumptions \ref{assumption2},  \ref{assumption3} and \ref{assumption4} be satisfied. If $X_0$ is an $\mathcal{F}_0$-measurable $H$-valued random variable, 
then there exists a unique mild solution $X$ giving by \eqref{mild1}.
Moreover the solution belongs to $\mathcal{C}\left([0, +\infty), \mathcal{C}(\overline{\Lambda}, \mathbb{R}\right)$
and for any $p\geq 1$, there exists a constant $C=C(p,T, X_0)>0$ such that 
\begin{eqnarray}
\label{prior1}
\sup_{t\in[0,T]}\mathbb{E}\left[\Vert X(t)\Vert^{2p}_{\mathcal{C}(\overline{\Lambda},\mathbb{R})}\right]\leq C.
\end{eqnarray}
\end{proposition}
\begin{proof}
The proof can be found in \cite[Theorem 7.7]{Prato} by taking $E=\mathcal{C}=\mathcal{C}(\overline{\Lambda}, \mathbb{R})$ and $H=L^2(\Lambda)$.  In fact, let $E^*$ be the topological dual of $E$ and $\langle .,.\rangle_{E^*}$  the duality pairing between $E^*$ and $E$. From \cite[Section Appendix A, Section A.5.2]{Prato} it holds that $-A$ generates a bounded  analytic semigroup $S(t)$ on $E$. Therefore from the Hille-Yosida theorem \cite[Theorem A.9 (i)]{Prato}  it follows that $\Vert e^{-tA_n}\Vert_{L(E)}\leq 1$, $t\geq 0$, where $A_n$, $n\in\mathbb{N}$ are the Yosida approximations of $A$ and $e^{-tA_n}$ is the semigroup generated by $-A_n$. Since $-A_n$, $n\in\mathbb{N}$ generates a contraction $C_0$-semigroup, it follows from \cite[Example D.8]{Prato} that $\langle -A_nu,u^*\rangle_{E^*}\leq 0$ for all $u\in E$ and $u^*\in \partial \Vert u\Vert_{\mathcal{C}}$, where $ \partial \Vert u\Vert_{\mathcal{C}}$ is the subdifferential of $\Vert u\Vert_{\mathcal{C}}$, defined in \cite[Appendix D]{Prato}.  Hence to prove that \cite[Assumption 7.4 (ii)]{Prato} is fulfilled, it is enough to prove that $\langle F(u+v), u^*\rangle_{E^*} \leq a(\Vert v\Vert_{\mathcal{C}})\left(1+\Vert u\Vert_{\mathcal{C}}\right)$,  for any $u, v\in E$ and $u^*\in \partial \Vert u\Vert_{\mathcal{C}}$. Note that to prove the later estimate it is enough to use the well-known characterization of the subdifferential of the norm in $E$, as in  \cite[Example 7.8]{Prato} or \cite[Examples D.7 and D.3]{Prato} and only prove that
\begin{eqnarray}
\label{dissis1}
\varphi(\xi+\eta)\leq\text{sgn}(\xi)a(\vert \eta\vert)\left(1+\vert\xi\vert\right),\quad \xi, \eta\in \mathbb{R}.
\end{eqnarray}
where $\text{sign}(\eta)=1$ if $\eta\geq 0$ and $\text{sign}(\eta)=-1$ if $\eta> 0$. Note that one can easily check that
\begin{eqnarray*}
\varphi(\xi+\eta)\leq \sup_{\alpha>0}(\alpha+\eta),\;\xi>0,\;\eta\in \mathbb{R} \quad \text{and}\quad 
-\varphi(\xi+\eta)\leq \sup_{\alpha<0}(\alpha+\eta),\; \xi<0,\, \eta\in \mathbb{R},
\end{eqnarray*} 
which prove that \eqref{dissis1} holds and hence \cite[Assumption 7.4 (ii)]{Prato} holds. Note that from \rmref{remarkassumption3} \cite[Assumption 7.4 (i)]{Prato} holds and obviously \cite[Assumption 7.3 (ii)]{Prato} holds. Finally applying \cite[Theorem 7.7 (ii)]{Prato} completes the proof of \propref{propexistence}. 
\end{proof}

\subsection{Fully discrete scheme and main result }
\label{numericalscheme}
Let us  first perform the space  approximation of  problem \eqref{model}.  We start by   discretizing  our domain $\Lambda$ in finite triangulation.
Let $\mathcal{T}_h$ be a triangulation with maximal length $h$. Let $V_h \subset V$ denote the space of continuous functions that are piecewise 
linear over the triangulation $\mathcal{T}_h$. We consider the projection $P_h$ and the discrete operator $A_h$  defined  respectively from  $L^2(\Lambda)$ to $V_h$ and from $V_h$ to $V_h$ by 
\begin{eqnarray}
\label{discreteop}
\langle P_hu,\chi\rangle=\langle u,\chi\rangle, \quad  \chi\in V_h,\, u\in H, \quad \text{and}\quad 
\langle A_h\phi,\chi\rangle=a(\phi,\chi),\quad  \phi,\chi\in V_h.
\end{eqnarray}
 The discrete operator $-A_h$   is also a generator 
of an analytic semigroup $S_h(t) : =e^{-tA_h}$. 
The semi-discrete version  in space  of problem \eqref{model} consists of finding $X^h(t)\in V_h$ such that 
\begin{eqnarray}
\label{semi1}
dX^h(t)+A_hX^h(t)dt=P_hF(X^h(t))dt+P_hdW(t), \quad X^h(0)=P_hX_0,\quad t\in(0,T].
\end{eqnarray}
Note that $S_h(t)$ and $P_hF$ satisfy the same assumptions as $S(t)$ and $F$ respectively, therefore as in \propref{propexistence}, \eqref{semi1} has a unique mild solution $X^h$, which  belongs to $\mathcal{C}\left((0, +\infty), \mathcal{C}(\overline{\Lambda}, \mathbb{R}\right)$
and for any $p\geq 1$, there exists a constant $C=C(p,T, X_0)>0$ such that 
\begin{eqnarray}
\label{prior2}
\sup_{t\in[0,T]}\mathbb{E}\left[\Vert X^h(t)\Vert^{2p}_{\mathcal{C}(\overline{\Lambda},\mathbb{R})}\right]\leq C,\quad h>0.
\end{eqnarray}
 The mild form of $X^h(t)$ is given as follows
\begin{eqnarray}
\label{mild2}
X^h(t)=S_h(t)X^h(0)+\int_0^tS_h(t-s)P_hF(X^h(s))ds+\int_0^tS_h(t-s)P_hdW(s),\quad t\in[0,T].
\end{eqnarray}
Let $t_m=m\Delta t\in[0,T]$, where  $\Delta t=T/M$ and $M\in\mathbb{N}$.  Applying the backward Euler method to \eqref{semi1} yields the following scheme
\begin{eqnarray}
\label{scheme1}
\left\{\begin{array}{ll}
X^h_0=P_hX_0,\\
X^h_{m+1}=S_{h, \Delta t}X^h_m+\Delta tS_{h,\Delta t}P_hF(X^h_{m+1})+S_{h,\Delta t}P_h\Delta W_m,\quad m=0,1,\cdots, M-1,
\end{array}
\right.
\end{eqnarray}
where $\Delta W_m$ and $S_{h,\Delta t}$ are given respectively  by
\begin{eqnarray*}
\Delta W_m:=W(t_{m+1})-W(t_m)\quad \text{and}\quad S_{h, \Delta t}:=(\mathbf{I}+\Delta tA_h)^{-1}.
\end{eqnarray*}
With the numerical method in hand, we can state our strong convergence result.

 Throughout this paper, $C$ is a positive constant independent of $h$, $m$, $M$ and $\Delta t$; that may change from one place to another.
 \begin{theorem}\textbf{[Main Result]}
\label{mainresult1}
Let  $X$ be  the mild  solution of   problem \eqref{model} and  $X^h_m$ be the numerical approximation defined in \eqref{scheme1}.
Let Assumptions    \ref{assumption1}, \ref{assumption2}, \ref{assumption3} and \ref{assumption4} be fulfilled. Then for any $p\geq 1$ the following error estimate holds 
\begin{eqnarray*}
 \Vert X(t_m)-\xi^h_m\Vert_{L^{2p}(\Omega,H)}\leq C\left(h^{\beta}+\Delta t^{\frac{\beta}{2}-\epsilon}\right),\quad m=0,1,\cdots,M.
 \end{eqnarray*}
 where $\epsilon$ is a positive number and $\beta$ is defined in \assref{assumption1}. 
\end{theorem}

\section{Proof of the main result}
\label{Convergence}
The proof of the main result need some preparatory results. 
\subsection{Preparatory results}
The following lemma will be useful for the space error estimate.
 \begin{lemma}  
 \label{RecallTambue}
 \begin{enumerate}
 \item[(i)] For any $r\in[0,2]$, it holds that
 \begin{eqnarray}
\label{addis1}
\Vert \left(S(t)-S_h(t)P_h\right)v\Vert \leq Ch^{r}t^{-(r-\alpha)/2}\Vert v\Vert_{\alpha},\quad  \alpha\leq r,\quad v\in \mathcal{D}\left(A^{\frac{\alpha}{2}}\right),\quad t>0.
\end{eqnarray}
\item[(ii)] For any $0\leq \gamma\leq 2$, it holds that
\begin{eqnarray*}
\left(\int_0^t\Vert \left(S(t)-S_h(t)P_h\right)v\Vert^2ds\right)^{1/2}\leq Ch^{\gamma}\Vert v\Vert_{\gamma-1},\quad v\in\mathcal{D}\left(A^{\frac{\gamma-1}{2}}\right),\quad t>0.
\end{eqnarray*}
\item[(iii)] Let $0\leq\rho\leq 1$. Then it holds that
\begin{eqnarray*}
\left\Vert\int_0^t\left(S(t)-S_h(t)P_h\right)vds\right\Vert\leq Ch^{2-\rho}\Vert v\Vert_{-\rho},\quad v\in\mathcal{D}\left(A^{-\rho}\right),\quad t>0.
\end{eqnarray*}
 \end{enumerate}
\end{lemma}
\begin{proof}
The proof of (i)-(ii) can be found in \cite[Lemma 6.1]{Antonio2}. The proof of (iii) can be found in \cite[Lemma 3.2 (iv)]{Antjd2}.
\end{proof}

The following lemma provides a version of \cite[Lemma 3.2 (iii)]{Stig1} for not necessary self-adjoint operator. 
\lemref{Sharpestimates} is key to achieve optimal regularity results of the mild solution \eqref{mild1}. This result will be  useful to achieve  optimal error estimate in space.
\begin{lemma}\cite[Lemma 2.1]{Antjd3}
\label{Sharpestimates}
For any  $0\leq\rho\leq 1$  and $0\leq\gamma\leq 2$, there exists a constant $C$ such that
\begin{eqnarray}
\label{raphael1}
\int_{t_1}^{t_2}\left\Vert A^{\rho/2}S(t_2-r)\right\Vert^2_{L(H)}dr\leq C(t_2-t_1)^{1-\rho},\quad 0\leq t_1\leq t_2\leq T,\\
\label{raphael2}
\int_{t_1}^{t_2}\left\Vert A^{\gamma/2}S(t_2-r)\right\Vert_{L(H)}dr\leq C(t_2-t_1)^{1-\gamma/2},\quad 0\leq t_1\leq t_2\leq T.
\end{eqnarray}
Note that \eqref{raphael1} and \eqref{raphael2} also hold when $A$ and $S(t)$ are replaced by their discrete approximations $A_h$ and $S_h(t)$ respectively.
\end{lemma}

\begin{lemma}
\label{embeddinglemma}
\begin{itemize}
\item[(i)] For any $\gamma\in[0, 2]$, the following norm are equivalent
\begin{eqnarray}
\label{embedd1}
\Vert A^{\gamma/2}v^h\Vert\approx\Vert A_h^{\gamma/2}v^h\Vert,\quad v^h\in \mathcal{D}\left(A^{\frac{\gamma}{2}}\right)\cap V_h.
\end{eqnarray}
\item[(ii)] Under Assumptions \ref{assumption3} and \ref{assumption4}, the following a priori estimates hold
\begin{eqnarray*}
\Vert A_h^{\beta/2}W_{A_h}(t)\Vert_{L^{2p}(\Omega, H)}\leq C,\;\left\Vert W_{A_h}(t)\right\Vert_{L^{2p}(\Omega, \mathcal{C})}\leq C,\; \left\Vert F(W_{A_h}(t))\right\Vert_{L^{2p}(\Omega, H)}\leq C,\;t\in[0, T], 
\end{eqnarray*}
where the stochastic convolution $W_{A_h}(t)$ is given by
\begin{eqnarray}
\label{bruitest1}
W_{A_h}(t):=\int_0^tS_h(t-s)P_hdW(s).
\end{eqnarray}
\item[(iii)] Under Assumptions \ref{assumption3} and \ref{assumption4}, the following a priori estimates hold
\begin{eqnarray*}
\Vert A^{\beta/2}W_{A}(t)\Vert_{L^{2p}(\Omega, H)}\leq C,\;\left\Vert W_{A}(t)\right\Vert_{L^{2p}(\Omega, \mathcal{C})}\leq C,\; \left\Vert F(W_{A}(t))\right\Vert_{L^p(\Omega, H)}\leq C,\;t\in[0, T], 
\end{eqnarray*}
where the stochastic convolution $W_{A}(t)$ is given by
\begin{eqnarray}
\label{bruitest1}
W_{A}(t):=\int_0^tS(t-s)dW(s).
\end{eqnarray}
\end{itemize}
\end{lemma}

\begin{proof}
Let us start with the proof of (i). Note that \eqref{embedd1} obviously holds for $\gamma=0$.  Using the equivalence of norms (see e.g., \cite{Larsson2})
\begin{eqnarray}
\label{rev1}
\Vert Av\Vert\approx \Vert v\Vert_2,\quad v\in V\cap H^2(\Lambda), 
\end{eqnarray}
and noticing that for $v^h\in \mathcal{D}(A)\cap V_h$, we have $A_h^{-1}v^h\in \mathcal{D}(A_h)\cap V_h\subset V\cap H^2(\Lambda)$, we obtain
\begin{eqnarray}
\label{rev2}
\left\Vert A A_h^{-1}v^h\right\Vert\leq C\Vert A_h^{-1}v^h\Vert_2,\quad v^h\in \mathcal{D}(A)\cap V_h\subset V\cap H^2(\Lambda).
\end{eqnarray}
Using the definition of $A_h$ \eqref{discreteop} and the Cauchy Schwartz's inequality, it follows that
\begin{eqnarray}
\label{rev4}
\Vert A_hv^h\Vert^2=\left\langle A_hv^h, A_hv^h\right\rangle=\left\langle Av^h, A_hv^h\right\rangle\leq \Vert Av^h\Vert\Vert A_hv^h\Vert.
\end{eqnarray}
It follows from \eqref{rev4} and \eqref{rev1} that
\begin{eqnarray}
\label{rev5}
\Vert A_hv^h\Vert\leq \Vert Av^h\Vert\leq C\Vert v^h\Vert_2,\quad v^h\in V_h\cap \mathcal{D}(A).
\end{eqnarray}
On the order hand, using the coercivity property \eqref{ellip2}, the definition of $A_h$ \eqref{discreteop} and Cauchy Schwartz's inequality, for $v^h\in V_h$, it holds that
\begin{eqnarray}
\label{rev6}
\Vert v^h\Vert^2_1&\leq& \frac{1}{\lambda_0}a(v^h,v^h)=\frac{1}{\lambda_0}\left\langle A_hv^h, v^h\right\rangle=\frac{1}{\lambda_0}\left\langle A_h^{1/2}v^h, (A_h^*)^{1/2}v^h\right\rangle\leq\frac{1}{\lambda_0}\left\Vert A_h^{1/2}v^h\right\Vert\left\Vert (A_h^*)^{1/2}v^h\right\Vert\nonumber\\
&\leq& \frac{1}{\lambda_0}\left\Vert A_h^{1/2}v^h\right\Vert\left\Vert (A_h^*)^{1/2}A_h^{-1/2}\right\Vert_{L(H)}\left\Vert A_h^{1/2}v^h\right\Vert\leq C\left\Vert A_h^{1/2}v^h\right\Vert^2,
\end{eqnarray}
where at the last step we employed \cite[Lemma 3.1]{Antonio2}. Thus from \eqref{rev6} we obtain
\begin{eqnarray}
\label{rev7}
\Vert v^h\Vert_1 \leq C\left\Vert A_h^{1/2}v^h\right\Vert,\quad v_h\in V_h.
\end{eqnarray}
Using the definition of the $H^2(\Lambda)$ norm and \eqref{rev7}, for any $v^h\in V_h\cap\mathcal{D}(A)\subset H^2(\Lambda)$, it holds that
\begin{eqnarray}
\label{rev8}
\Vert v^h\Vert_2^2=\Vert v^h\Vert^2+\left\Vert\nabla v^h\right\Vert_1^2\leq \left\Vert v^h\right\Vert^2+C\left\Vert A_h^{1/2}\nabla v^h\right\Vert^2=\Vert v^h\Vert^2+C\left\Vert \nabla \left(A_h^{1/2}v^h\right)\right\Vert^2,
\end{eqnarray}
where at the last step we have  used the fact that $A_h^{1/2}$ is a linear operator and thus commutes with weak derivatives, see e.g., \cite[(26)]{Antjd1}. Using again \eqref{rev7} yields
\begin{eqnarray}
\label{rev9}
\Vert v^h\Vert_2^2&\leq& \Vert v^h\Vert^2+\left\Vert A_h^{1/2}v^h\right\Vert_1^2\leq \Vert v^h\Vert^2+C\left\Vert A_h^{1/2}A_h^{1/2}v^h\right\Vert^2\nonumber\\
&\leq& \Vert v^h\Vert^2+C\left\Vert A_hv^h\right\Vert^2\leq C'\left\Vert A_hv^h\right\Vert^2.
\end{eqnarray}
Combining \eqref{rev9} and \eqref{rev5} yields
\begin{eqnarray}
\label{rev10}
\Vert v^h\Vert_2\approx \Vert A_hv^h\Vert,\quad v^h\in V_h\cap \mathcal{D}(A).
\end{eqnarray}
Combining \eqref{rev10} and \eqref{rev1} yields
\begin{eqnarray}
\label{rev11}
\Vert Av^h\Vert\approx\Vert A_hv^h\Vert,\quad v\in \mathcal{D}\left(A\right)\cap V_h.
\end{eqnarray}
Hence \eqref{rev11} proved \eqref{embedd1} for $\gamma=2$. The intermediate cases of \eqref{embedd1} are completed by interpolation theory. Let us now move to the proof of (ii). 
Using the Burkholder-Davis-Gundy inequality \eqref{Davis1}, \cite[Lemma 11]{Antjd1}, and \lemref{Sharpestimates}, it holds that
\begin{eqnarray}
\label{bruitest3}
\left\Vert A_h^{\beta/2}W_{A_h}(t)\right\Vert_{L^{2p}(\Omega, H)}&\leq& C(p)\left(\int_0^t\left\Vert A_h^{\beta/2}S_h(t-s)P_hQ^{\frac{1}{2}}\right\Vert^2_{L^{2p}\left(\Omega,\mathcal{L}_2(H)\right)}ds\right)^{1/2}\nonumber\\
&\leq& C\left(\int_0^t\left\Vert S_h(t-s)A_h^{\frac{1}{2}}\right\Vert^2_{L(H)}\left\Vert A_h^{\frac{\beta-1}{2}}P_hQ^{\frac{1}{2}}\right\Vert^2_{\mathcal{L}_2(H)}ds\right)^{1/2}\nonumber\\
&\leq& C\left(\int_0^t\left\Vert S_h(t-s)A_h^{\frac{1}{2}}\right\Vert^2_{L(H)}ds\right)^{1/2}\leq C.
\end{eqnarray}
 Using the boundedness of $A_h^{-\beta/2}$ and  \eqref{bruitest3}, it holds that 
\begin{eqnarray}
\label{bruitest4}
\Vert W_{A_h}(t)\Vert_{L^{2p}(\Omega, H)}\leq \Vert A_h^{-\beta/2}\Vert_{L(H)}\Vert A_h^{\beta/2}W_{A_h}(t)\Vert_{L^{2p}(\Omega, H)}\leq C.
\end{eqnarray}
Using the Sobolev embbeding \eqref{embedd0a}, \eqref{embedd1} and \eqref{bruitest3}, it follows that $W_{A_h}(t)\in L^{2p}(\Omega,\mathcal{C})$ and
\begin{eqnarray}
\label{bruitest5}
\mathbb{E}\left[\Vert W_{A_h}(t)\Vert^{2p}_{\mathcal{C}}\right]\leq C\mathbb{E}\left[\Vert A^{\beta/2}W_{A_h}(t)\Vert^{2p}\right]\leq C\mathbb{E}\left[\Vert A_h^{\beta/2}W_{A_h}(t)\Vert^{2p}\right]\leq C,\quad t\in[0, T].
\end{eqnarray}
Using \assref{assumption3},  Cauchy-Schwartz's inequality, \eqref{bruitest4} and \eqref{bruitest5} yields
\begin{eqnarray*}
\Vert F\left(W_{A_h}(t)\right)\Vert_{L^{2p}(\Omega, H)}&\leq& \Vert W_{A_h}(t)\Vert_{L^{4p}(\Omega, H)}\left(1+\mathbb{E}\left[\Vert W_{A_h}(t)\Vert^{4pc_1}_{\mathcal{C}}\right]\right)^{1/2}\leq C,\; t\in[0, T].
\end{eqnarray*} 
This completes the proof of (ii). The proof of (iii) is similar to that of (ii).
\end{proof}

\begin{theorem}
\label{Regularity}
Let Assumptions \ref{assumption1}, \ref{assumption2}, \ref{assumption3} and \ref{assumption4} be fulfilled. Then for any $p\geq 1$, the following space regularity holds for the semi-discrete solution $X^h(t)$
\begin{eqnarray}
\label{spaceregular}
\Vert A_h^{\beta/2}X^h(t)\Vert_{L^{2p}(\Omega, H)}\leq C,\quad \Vert X^h(t)\Vert_{L^{2p}(\Omega, \mathcal{C})}\leq C, \quad \Vert F(X^h(t))\Vert_{L^{2p}(\Omega, H)}\leq C,\quad t\in[0, T],
\end{eqnarray}
where $\beta$ comes from \assref{assumption1}. The following time regularity also holds
\begin{eqnarray}
\label{timeregular}
\Vert X^h(t)-X^h(s)\Vert_{L^{2p}(\Omega, H)}\leq C(t-s)^{\min\left(\frac{1}{2}, \frac{\beta}{2}\right)},\quad 0\leq s\leq t\leq T.
\end{eqnarray}
Note that the above space and time regularities still holds when $X^h(t)$ is replaced by $X(t)$.
\end{theorem}
\begin{proof}
We start by proving  that 
\begin{eqnarray*}
\Vert X^h(t)\Vert_{L^p(\Omega, H)}\leq C,\quad t\in[0, T].
\end{eqnarray*}
 Let us introduce the following auxiliary  process
\begin{eqnarray*}
Y^h(t):=X^h(t)-W_{A_h}(t)=S_h(t)P_hX_0+\int_0^tS_h(t-s)P_hF(X^h(s))ds,\quad t\in[0,T].
\end{eqnarray*}
Then it follows that $Y^h(t)$ is differentiable with respect to the time and is the solution of the following deterministic problem
\begin{eqnarray}
\label{sam1}
\frac{d}{dt}Y^h(t)+A_hY^h(t)=P_hF\left(Y^h(t)+W_{A_h}(t)\right),\quad Y^h(0)=P_hX_0,\quad t\in(0,T].
\end{eqnarray}
Taking the inner product in both sides of \eqref{sam1}, using \assref{assumption2} and Cauchy-Schwartz's inequality yields
\begin{eqnarray}
\label{sam2}
\frac{1}{2}\frac{d}{ds}\left\Vert Y^h(s)\right\Vert^2+\left\langle A_hY^h(s), Y^h(s)\right\rangle&=&\left\langle F\left(Y^h(s)+W_{A_h}(s)\right)-F(W_{A_h}), Y^h(s)\right\rangle\nonumber\\
&+&\left\langle F(W_{A_h}(s)), Y^h(s)\right\rangle\nonumber\\
&\leq& C\left\Vert Y^h(s)\right\Vert^2+\frac{1}{2}\left\Vert Y^h(s)\right\Vert^2+\frac{1}{2}\left\Vert F\left(W_{A_h}(s)\right)\right\Vert^2.
\end{eqnarray}
Using the coercivity estimate \eqref{ellip2},  the fact that $Y^h(s)\in V_h$ and \eqref{sam2} yields
\begin{eqnarray}
\label{sam3}
\lambda_0\left\Vert Y^h(s)\right\Vert^2_1&\leq& a\left(Y^h(s), Y^h(s)\right)\nonumber\\
&=&-\frac{1}{2}\frac{d}{ds}\left\Vert Y^h(s)\right\Vert^2+\frac{1}{2}\frac{d}{ds}\left\Vert Y^h(s)\right\Vert^2+\left\langle A_hY^h(s), Y^h(s)\right\rangle\nonumber\\
&\leq& -\frac{1}{2}\frac{d}{ds}\left\Vert Y^h(s)\right\Vert^2+C\left\Vert Y^h(s)\right\Vert^2+\frac{1}{2}\left\Vert F(W_{A_h}(s))\right\Vert^2.
\end{eqnarray}
Since  $Y^h(s)\in V_h$, using Cauchy-Schwartz's inequality,   the equivalence of norms \cite[(2.12)]{Larsson2},  \cite[Lemma 3.1]{Antonio2} and \eqref{sam3}, it holds that
\begin{eqnarray}
\label{sam4}
\left\vert\left\langle A_hY^h(s), Y^h(s)\right\rangle\right\vert&=&\left\vert\left\langle A_h^{1/2}Y^h(s), (A_h^*)^{1/2}Y^h(s)\right\rangle\right\vert\leq \frac{1}{2}\left\Vert A_h^{1/2}Y^h(s)\right\Vert^2+\frac{1}{2}\left\Vert (A_h^*)^{1/2}Y^h(s)\right\Vert^2\nonumber\\
&\leq& \frac{C}{2}\left\Vert Y^h(s)\right\Vert^2_1+\frac{1}{2}\left\Vert (A_h^*)^{1/2}A_h^{-1/2}\right\Vert^2_{L(H)}\left\Vert A_h^{1/2}Y^h(s)\right\Vert^2\nonumber\\
&\leq& \frac{C}{2}\left\Vert Y^h(s)\right\Vert^2_1+\frac{C}{2}\left\Vert A_h^{1/2}Y^h(s)\right\Vert^2\leq C\left\Vert Y^h(s)\right\Vert^2_1=\frac{C}{\lambda_0}.\lambda_0\left\Vert Y^h(s)\right\Vert^2_1\nonumber\\
&\leq& -\frac{C}{2\lambda_0}\frac{d}{ds}\left\Vert Y^h(s)\right\Vert^2+C\left\Vert Y^h(s)\right\Vert^2+C\left\Vert F(W_{A_h}(s))\right\Vert^2.
\end{eqnarray}
Note that from \eqref{sam1} we have
\begin{eqnarray}
\label{sam5}
\frac{1}{2}\frac{d}{ds}\Vert Y^h(s)\Vert^2&\leq& C\Vert Y^h(s)\Vert^2+\frac{1}{2}\Vert F(W_{A_h}(s))\Vert^2-\left\langle A_hY^h(s), Y^h(s)\right\rangle\nonumber\\
&\leq& C\Vert Y^h(s)\Vert^2+\frac{1}{2}\Vert F(W_{A_h}(s))\Vert^2+\left\vert\left\langle A_hY^h(s), Y^h(s)\right\rangle\right\vert.
\end{eqnarray}
Substituting \eqref{sam4} in \eqref{sam5} yields
\begin{eqnarray}
\label{sam6}
\left(\frac{1}{2}+\frac{C}{2\lambda_0}\right)\frac{d}{ds}\Vert Y^h(s)\Vert^2\leq  C\Vert Y^h(s)\Vert^2+C\Vert F(W_{A_h}(s))\Vert^2.
\end{eqnarray}
Integrating both sides of \eqref{sam6} over $[0, t]$ and using \lemref{embeddinglemma} (ii) yields
\begin{eqnarray}
\label{sam7}
\Vert Y^h(t)\Vert^2_{L^{2p}(\Omega, H)}&\leq& C\int_0^t\Vert Y^h(s)\Vert^2_{L^{2p}(\Omega, H)}ds+C\int_0^t\Vert F(W_{A_h}(s))\Vert^2_{L^{2p}(\Omega, H)}ds\nonumber\\
&\leq& C\int_0^t\Vert Y^h(s)\Vert^2_{L^{2p}(\Omega, H)}ds+C.
\end{eqnarray}
Applying the  Gronwall's lemma to \eqref{sam7} yields $\Vert Y^h(t)\Vert^2_{L^{2p}(\Omega, H)}\leq C$. Using triangle inequality and \lemref{embeddinglemma} (ii), it follows that
\begin{eqnarray}
\label{sam7a}
\Vert X^h(t)\Vert_{L^{2p}(\Omega, H)}\leq \Vert Y^h(t)\Vert_{L^{2p}(\Omega, H)}+\Vert W_{A_h}(t)\Vert_{L^{2p}(\Omega, H)}\leq C,\quad t\in[0,T].
\end{eqnarray}
Employing \rmref{remarkassumption3}, Cauchy-Schwartz's inequality, \eqref{prior2} and \eqref{sam7a}, it follows that
\begin{eqnarray}
\label{sam7b}
\Vert F(X^h(t))\Vert_{L^{2p}(\Omega, H)}\leq C\Vert X^h(t)\Vert_{L^{4p}(\Omega, H)}\left(1+\sup_{t\in[0,T]}\mathbb{E}\left[\Vert X^h(t)\Vert^{4pc_1}_{\mathcal{C}(\overline{\Lambda}, \mathbb{R})}\right]\right)^{1/4p}\leq C.
\end{eqnarray} 
Let us now move to the estimate of \eqref{spaceregular}.
From the mild solution \eqref{mild2}, and using triangle inequality it follows that
\begin{eqnarray}
\label{space1}
\left\Vert A_h^{\beta/2}X(t)\right\Vert_{L^{2p}(\Omega, H)}&\leq& \left\Vert S_h(t)A_h^{\beta/2}X_0\right\Vert_{L^{2p}(\Omega, H)}+\left\Vert\int_0^tS_h(t-s)A_h^{\beta/2}P_hF(X^h(s))ds\right\Vert_{L^{2p}(\Omega, H)}\nonumber\\
&+&\left\Vert\int_0^tS_h(t-s)A_h^{\beta/2}P_hdW(s)\right\Vert_{L^{2p}(\Omega, H)}=:I_1+I_2+I_3.
\end{eqnarray}
Employing \cite[Lemma 1]{Antjd1} and \assref{assumption1}, it holds that
\begin{eqnarray}
\label{space2}
I_1\leq \Vert S_h(t)\Vert_{L(H)}\left\Vert A^{\beta/2}X_0\right\Vert_{L^{2p}(\Omega, H)}\leq C.
\end{eqnarray}
Using \propref{propexistence}, \eqref{spaceregular}  and \lemref{Sharpestimates}, it holds that
\begin{eqnarray}
\label{space3}
I_2\leq C\int_0^t\left\Vert S_h(t-s)A_h^{\beta/2}\right\Vert_{L(H)}\Vert F(X^h(s))\Vert_{L^{2p}(\Omega, H)}ds\leq C.
\end{eqnarray}
Using the Burkh\"{o}lder-Davis-Gundy's inequality \eqref{Davis1} and \lemref{Sharpestimates}, it holds that
\begin{eqnarray}
\label{space4}
I_3&\leq& \left(\int_0^t\left\Vert S_h(t-s)A_hP_hQ^{\frac{1}{2}}\right\Vert^2_{L^{2p}\left(\Omega,\mathcal{L}_2(H)\right)}ds\right)^{1/2}\nonumber\\
&\leq& \left(\int_0^t\left\Vert S_h(t-s)A_h^{\frac{\beta-1}{2}}\right\Vert_{L(H)}^2\left\Vert A_h^{\frac{\beta-1}{2}}P_hQ^{\frac{1}{2}}\right\Vert^2_{\mathcal{L}_2(H)}ds\right)^{1/2}\nonumber\\
&\leq& C\left(\int_0^t\left\Vert S_h(t-s)A_h^{\frac{\beta-1}{2}}\right\Vert^2_{L(H)}ds\right)^{1/2}\leq C.
\end{eqnarray}
Substituting \eqref{space4}, \eqref{space3} and \eqref{space2} in \eqref{space1} completes the proof of \eqref{spaceregular}. We are left with the proof of \eqref{timeregular}. Using the mild representation of the solution of \eqref{semi1}, it holds that
\begin{eqnarray}
\label{time1}
X^h(t)-X^h(s)=\left(S_h(t-s)-\mathbf{I}\right)X^h(s)+\int_s^tS_h(t-r)P_hF(X^h(\sigma))d\sigma+\int_s^tS_h(t-\sigma)P_hdW(\sigma).
\end{eqnarray}
Using triangle inequality, it follows from \eqref{time1} that
\begin{eqnarray}
\label{time2}
\Vert X^h(t)-X^h(s)\Vert_{L^{2p}(\Omega, H)}&=&\left\Vert\left(S_h(t-s)-\mathbf{I}\right)X^h(s)\right\Vert_{L^{2p}(\Omega, H)}\nonumber\\
&+&\left\Vert\int_s^tS_h(t-\sigma)P_hF(X^h(\sigma))d\sigma\right\Vert_{L^{2p}(\Omega, H)}\nonumber\\
&+&\left\Vert\int_s^tS_h(t-\sigma)P_hdW(\sigma)\right\Vert_{L^{2p}(\Omega, H)}=:II_1+II_2+II_3.
\end{eqnarray}
Using the smoothing properties of the semigroup and \eqref{spaceregular}, it holds that
\begin{eqnarray}
\label{time3}
II_1\leq \left\Vert\left(S_h(t-s)-\mathbf{I}\right)A_h^{-\beta/2}\right\Vert_{L(H)}\left\Vert A_h^{\beta/2}X^h(s)\right\Vert_{L^{2p}(\Omega, H)}\leq C(t-s)^{\beta/2}.
\end{eqnarray}
Using the smoothing properties of the semigroup and \eqref{spaceregular}, it holds that
\begin{eqnarray}
\label{time4}
II_2\leq \int_s^t\Vert S_h(t-\sigma)P_hF(X^h(\sigma))\Vert_{L^{2p}(\Omega, H)}d\sigma\leq C(t-s).
\end{eqnarray}
Using the Burkholder-Davis-Gundy inequality \eqref{Davis1} and \lemref{Sharpestimates}, it holds that
\begin{eqnarray}
\label{time5}
II_3&\leq& \left(\int_s^t\left\Vert S_h(t-\sigma)Q^{\frac{1}{2}}\right\Vert^2_{L^{2p}(\Omega, \mathcal{L}_2(H))}d\sigma\right)^{1/2}\nonumber\\
&\leq& \left(\int_s^t\left\Vert S_h(t-r)A_h^{\frac{1-\beta}{2}}\right\Vert^2_{L(H)}\left\Vert A_h^{\frac{\beta-1}{2}}Q^{\frac{1}{2}}\right\Vert^2_{\mathcal{L}_2(H)}d\sigma\right)^{1/2}\leq C(t-s)^{\min\left(\frac{1}{2}, \frac{\beta}{2}\right)}.
\end{eqnarray}
Substituting \eqref{time5}, \eqref{time4} and \eqref{time2} in \eqref{time2} completes the proof of \eqref{timeregular}.
\end{proof}


\begin{theorem}\textbf{[Space error]} 
\label{Spaceerror}
Let $X(t)$ and $X^h(t)$ be solution of \eqref{model} and \eqref{semi1} respectively. Let Assumptions \ref{assumption1}, \ref{assumption2}, \ref{assumption3} and \ref{assumption4} be fulfilled. Then for any $p\geq 1$ the following  holds
\begin{eqnarray}
\Vert X(t)-X^h(t)\Vert_{L^{2p}(\Omega, H)}\leq Ch^{\beta},\quad t\in[0, T].
\end{eqnarray}
\end{theorem}
\begin{proof}
Let us introduce the following auxiliary process
\begin{eqnarray}
\label{aux1}
\widetilde{X}^h(t):=S_h(t)P_hX_0+\int_0^tS_h(t-s)P_hF(X(s))ds+W_{A_h}(t).
\end{eqnarray}
Using triangle inequality, it holds that
\begin{eqnarray}
\label{aux1a}
\Vert X(t)-X^h(t)\Vert_{L^{2p}(\Omega, H)}\leq \Vert X(t)-\widetilde{X}^h(t)\Vert_{L^{2p}(\Omega, H)}+\Vert \widetilde{X}^h(t)-X^h(t)\Vert_{L^{2p}(\Omega, H)}.
\end{eqnarray}
Using \thmref{Regularity}, \eqref{prior1} and \assref{assumption2}, one can easily check that
\begin{eqnarray}
\label{aux1b}
\Vert \widetilde{X}^h(t)\Vert_{L^{2p}(\Omega, H)}<C,\quad \sup_{t\in[0, T]}\mathbb{E}\left[\Vert\widetilde{X}^h(t)\Vert^{2p}_{\mathcal{C}(\overline{\Lambda}, \mathbb{R})}\right]\leq C,\quad \Vert F(\widetilde{X}^h(t))\Vert_{L^{2p}(\Omega, H)}\leq C,\quad t\in[0, T].
\end{eqnarray}
Let us start by estimating the first term in \eqref{aux1a}. From \eqref{aux1} and \eqref{mild1} and using triangle inequality, it holds that
\begin{eqnarray}
\label{aux1c}
\Vert X(t)-\widetilde{X}^h(t)\Vert_{L^{2p}(\Omega, H)}&\leq&\left\Vert\left(S(t)-S_h(t)P_h\right)X_0\right\Vert_{L^{2p}(\Omega, H)}\nonumber\\
&+&\left\Vert\int_0^t\left(S(t-s)-S_h(t-s)P_h\right)F(X(s))ds\right\Vert_{L^{2p}(\Omega, H)}\nonumber\\
&=:&III_1+III_2.
\end{eqnarray}
Using \lemref{RecallTambue} (i) with $r=\alpha=\beta$ and  \assref{assumption1} for the estimate of $III_1$ and \lemref{RecallTambue} (iii) with $\rho=0$ for the estimate of $III_2$  yields
\begin{eqnarray}
\label{aux1d}
III_1\leq Ch^{\beta}\Vert A^{\beta/2}X_0\Vert_{L^{2p}(\Omega, H)}\leq Ch^{\beta},\quad III_2\leq Ch^{2}.
\end{eqnarray}
Substituting \eqref{aux1d} in \eqref{aux1c} yields
\begin{eqnarray}
\label{errorintermediare}
\Vert X(t)-\widetilde{X}^h(t)\Vert_{L^{2p}(\Omega, H)}\leq Ch^{\beta},\quad t\in[0, T].
\end{eqnarray}
Let us introduce the following error representation $\widetilde{e}^h(t):=\widetilde{X}^h(t)-X^h(t)$. It is immediate to see that $\widetilde{e}^h(t)$ is  differentiable with repect to the time  and satisfies 
\begin{eqnarray}
\label{aux2}
\frac{d}{dt}\widetilde{e}^h(t)+A_h\widetilde{e}^h(t)=P_h\left(F(X(t))-F(X^h(t))\right),\quad t\in(0, T],\quad \widetilde{e}^h(0)=0.
\end{eqnarray}
Taking the inner product in both sides of \eqref{aux2}, using \assref{assumption2} and Cauchy-Schwartz's inequality yields
\begin{eqnarray}
\label{aux3}
\frac{1}{2}\frac{d}{ds}\Vert\widetilde{e}^h(s)\Vert^2+\left\langle A_h\widetilde{e}^h(s), \widetilde{e}^h(s)\right\rangle&=&\left\langle F(\widetilde{X}^h(s))-F(X^h(s)), \widetilde{e}^h(s)\right\rangle\nonumber\\
&+&\left\langle F(X(s))-F(\widetilde{X}^h(s)), \widetilde{e}^h(s)\right\rangle\nonumber\\
&\leq& \Vert\widetilde{e}^h(s)\Vert^2+\frac{1}{2}\Vert F(X(s))-F(\widetilde{X}^h(s))\Vert^2+\frac{1}{2}\Vert\widetilde{e}^h(s)\Vert^2.
\end{eqnarray}
Using the coercivity estimate \eqref{ellip2},  the fact that $\widetilde{e}^h(s)\in V_h$ and \eqref{aux3} yields
\begin{eqnarray}
\label{aux4}
\lambda_0\Vert\widetilde{e}^h(s)\Vert^2_1&\leq& a\left(\widetilde{e}^h(s), \widetilde{e}^h(s)\right)\nonumber\\
&=&-\frac{1}{2}\frac{d}{ds}\Vert \widetilde{e}^h(s)\Vert^2+\frac{1}{2}\frac{d}{ds}\Vert \widetilde{e}^h(s)\Vert^2+\left\langle A_h\widetilde{e}^h(s), \widetilde{e}^h(s)\right\rangle\nonumber\\
&\leq& -\frac{1}{2}\frac{d}{ds}\Vert \widetilde{e}^h(s)\Vert^2+\frac{3}{2}\Vert \widetilde{e}^h(s)\Vert^2+\frac{1}{2}\Vert F(X(s))-F(\widetilde{X}^h(s))\Vert^2.
\end{eqnarray}
Using again \eqref{aux3}, it holds that
\begin{eqnarray}
\label{aux5}
\frac{1}{2}\frac{d}{ds}\Vert \widetilde{e}^h(s)\Vert^2&\leq& \frac{3}{2}\Vert\widetilde{e}^h(s)\Vert^2+\frac{1}{2}\Vert F(X(s))-F(\widetilde{X}^h(s))\Vert^2-\left\langle A_h\widetilde{e}^h(s), \widetilde{e}^h(s)\right\rangle\nonumber\\
&\leq& \frac{3}{2}\Vert\widetilde{e}^h(s)\Vert^2+\frac{1}{2}\Vert F(X(s))-F(\widetilde{X}^h(s))\Vert^2+\left\vert\left\langle A_h\widetilde{e}^h(s), \widetilde{e}^h(s)\right\rangle\right\vert.
\end{eqnarray}
Since  $\widetilde{e}^h(s)\in V_h$, using Cauchy-Schwartz's inequality,   the equivalence of norms \cite[(2.12)]{Larsson2} and \cite[Lemma 3.1]{Antonio2}, it holds that
\begin{eqnarray}
\label{aux6}
\left\vert\left\langle A_h\widetilde{e}^h(s), \widetilde{e}^h(s)\right\rangle\right\vert&=&\left\vert\left\langle A_h^{1/2}\widetilde{e}^h(s), (A_h^*)^{1/2}\widetilde{e}^h(s)\right\rangle\right\vert\leq \frac{1}{2}\Vert A_h^{1/2}\widetilde{e}^h(s)\Vert^2+\frac{1}{2}\Vert (A_h^*)^{1/2}\widetilde{e}^h(s)\Vert^2\nonumber\\
&\leq& \frac{C}{2}\Vert \widetilde{e}^h(s)\Vert^2_1+\frac{1}{2}\Vert (A_h^*)^{1/2}A_h^{-1/2}\Vert^2_{L(H)}\Vert A_h^{1/2}\widetilde{e}^h(s)\Vert^2\nonumber\\
&\leq& \frac{C}{2}\Vert \widetilde{e}^h(s)\Vert^2_1+\frac{C}{2}\Vert A_h^{1/2}\widetilde{e}^h(s)\Vert^2\leq C\Vert \widetilde{e}^h(s)\Vert^2_1.
\end{eqnarray}
Substituting \eqref{aux6} in \eqref{aux5} and employing \eqref{aux4} yields
\begin{eqnarray}
\label{aux7}
\frac{1}{2}\frac{d}{ds}\Vert \widetilde{e}^h(s)\Vert^2&\leq& \frac{3}{2}\Vert\widetilde{e}^h(s)\Vert^2+\frac{1}{2}\Vert F(X(s))-F(\widetilde{X}^h(s))\Vert^2 +\frac{C}{\lambda_0}.\lambda_0\Vert \widetilde{e}^h(s)\Vert^2_1\nonumber\\
&\leq& C\Vert\widetilde{e}^h(s)\Vert^2+C\Vert F(X(s))-F(\widetilde{X}^h(s))\Vert^2 -\frac{C}{2\lambda_0}\frac{d}{ds}\Vert \widetilde{e}^h(s)\Vert^2.
\end{eqnarray}
Therefore, from \eqref{aux7} we have 
\begin{eqnarray}
\label{aux8}
\left(\frac{1}{2}+\frac{C}{2\lambda_0}\right)\frac{d}{ds}\Vert \widetilde{e}^h(s)\Vert^2\leq  C\Vert\widetilde{e}^h(s)\Vert^2+C\Vert F(X(s))-F(\widetilde{X}^h(s))\Vert^2.
\end{eqnarray}
Integrating both sides of \eqref{aux8} over $[0, t]$ yields
\begin{eqnarray}
\label{aux9}
\Vert\widetilde{e}^h(t)\Vert^2\leq C\int_0^t\Vert\widetilde{e}^h(s)\Vert^2ds+C\int_0^t\Vert F(X(s))-F(\widetilde{X}^h(s))\Vert^2ds.
\end{eqnarray}
Using the Cauchy-Schwartz's inequality, \assref{assumption3}, \propref{propexistence}, \eqref{aux1b} and \eqref{errorintermediare} yields
\begin{eqnarray}
\label{aux10}
\Vert\widetilde{e}^h(t)\Vert^2_{L^{2p}(\Omega, H)}&\leq& C\int_0^t\Vert\widetilde{e}^h(s)\Vert^2_{L^{2p}(\Omega, H)}ds+C\int_0^t\Vert F(X(s))-F(\widetilde{X}^h(s))\Vert^2_{L^{2p}(\Omega, H)}ds\nonumber\\
&\leq& C\int_0^t\Vert\widetilde{e}^h(s)\Vert^2_{L^{2p}(\Omega, H)}ds\nonumber\\
&+&C\int_0^t\Vert X(s)-\widetilde{X}^h(s)\Vert^2_{L^{4p}(\Omega, H)}\nonumber\\
&\times&\left(1+\mathbb{E}\left[\sup_{s\in[0, T]}\Vert X(s)\Vert^{4pc_1}_{\mathcal{C}(\overline{\Lambda}, \mathbb{R})}\right]+\mathbb{E}\left[\sup_{s\in[0,T]}\Vert \widetilde{X}^h(s)\Vert^{4pc_1}_{\mathcal{C}(\overline{\Lambda},\mathbb{R})}\right]\right)^{1/2p}ds\nonumber\\
&\leq& C\int_0^t\Vert\widetilde{e}^h(s)\Vert^2_{L^{2p}(\Omega, H)}ds+Ch^{2\beta}.
\end{eqnarray}
Applying the Gronwall's lemma to \eqref{aux10} yields
\begin{eqnarray}
\label{aux11}
\Vert \widetilde{e}^h(t)\Vert_{L^{2p}(\Omega, H)}\leq Ch^{\beta},\quad t\in[0, T].
\end{eqnarray}
This completes the proof of the theorem. 
\end{proof}

Let us recall the following useful result from \cite[Lemma 3.3]{Antjd2}.
\begin{lemma} 
\label{RecallMukam}
\begin{itemize}
\item[(i)] For any $m$, $h$ and $\Delta t$ the following estimate holds
\begin{eqnarray*}
\Vert (\mathbf{I}+\Delta tA_h)^{-m}\Vert_{L(H)}\leq 1.
\end{eqnarray*}
\item[(ii)] Let  $0\leq q\leq 1$. Then for all $u \in H$, the following estimate holds
\begin{eqnarray*}
\Vert (S_h(t_m)-S_{h,\Delta t}^m)P_hu\Vert\leq C\Delta t^{q}t_m^{-q}\Vert u\Vert,\quad m=1,\cdots, M.
\end{eqnarray*}
\item[(iii)] For all $u\in\mathcal{D}(A^{\gamma-1})$, $0\leq\gamma\leq 2$, the following estimate holds
\begin{eqnarray*}
\Vert (S_{h,\Delta t}^m-S_h(t_m))P_hu\Vert\leq Ct_m^{-1/2}\Delta t^{\gamma/2}\Vert u\Vert_{\gamma-1}.
\end{eqnarray*}
\item[(iv)]
If $u\in\mathcal{D}(A^{\mu/2})$, $0\leq\mu\leq2$. Then the following error estimate holds 
\begin{eqnarray*}
\Vert (S_h(t_m)-S_{h,\Delta t}^m)P_hu\Vert\leq C\Delta t^{\mu/2}\Vert u\Vert_{\mu},\quad m=0,\cdots, M.
\end{eqnarray*}
\item[(v)] For  $u\in\mathcal{D}(A^{-\sigma})$, $0\leq\sigma\leq1$, the following estimate holds
\begin{eqnarray*}
\Vert (S_{h,\Delta t}^m-S_h(t_m))P_hu\Vert\leq Ct_m^{-1}\Delta t^{(2-\sigma)/2}\Vert u\Vert_{-\sigma}.
\end{eqnarray*}
\end{itemize}
\end{lemma}

Let us introduce the following process
\begin{eqnarray}
\label{reaux1a}
W_{A_h}^m:=\sum_{j=0}^{m-1}S^{m-j}_{h,\Delta t}P_h\Delta W_m, \quad m\in\{0, \cdots,M\},\quad M\in\mathbb{N}.
\end{eqnarray}

\begin{lemma}
\label{prioritylemma}
 Let Assumptions \ref{assumption2}, \ref{assumption3} and \ref{assumption4} be fulfilled. Then the following estimates hold
\begin{eqnarray}
\label{reaux1aa}
\Vert A_h^{\beta/2} W_{A_h}^m\Vert_{L^{2p}(\Omega, H)}\leq C \quad \text{and}\quad \Vert F(W_{A_h}^m)\Vert_{L^{2p}(\Omega, H)}\leq C, \; m=0, 1, \cdots M, \quad M\in \mathbb{N}.
\end{eqnarray}
\end{lemma}

\begin{proof}
 Let us introduce $S_{h, \Delta t}(t):=S_{h, \Delta t}^m$ for any  $t\in[t_{m-1}, t_m)$ and let $\chi_B$ be the characteristic function of $B\subset \mathbb{R}$. Then $W_{A_h}^m$ can be written as follows  
\begin{eqnarray}
\label{reaux1aaa}
W_{A_h}^m=\int_0^T\chi_{[0, t_m)}(s)S_{h, \Delta t}(t_m-s)P_hdW(s).
\end{eqnarray}
Using the Burkh\"{o}lder-Davis-Gundy's inequality \eqref{Davis1}, \cite[Lemma 11]{Antjd1} and \cite[Lemma 3.5 (ii)]{Antjd2}, it follows that
\begin{eqnarray}
\label{reaux1ab}
\Vert A_h^{\frac{\beta}{2}}W_{A_h}^m\Vert_{L^{2p}(\Omega, H)}&\leq& C\left(\int_0^T\left\Vert\chi_{[0, t_m)}(s)A_h^{\frac{\beta}{2}} S_{h,\Delta t}(t_m-s)P_hQ^{\frac{1}{2}}\right\Vert^2_{L^{2p}\left(\Omega, \mathcal{L}_2(H)\right)}ds\right)^{1/2}\nonumber\\
&\leq& C\left(\sum_{j=0}^{m-1}\int_{t_j}^{t_{j+1}}\Vert A_h^{\frac{\beta}{2}}S_{h, \Delta t}^{m-j}P_hQ^{\frac{1}{2}}\Vert^2_{\mathcal{L}_2(H)}ds\right)^{1/2}\nonumber\\
&\leq& C\left(\sum_{j=0}^{m-1}\int_{t_j}^{t_{j+1}}\Vert \left(S_{h, \Delta t}^{m-j}-S_h(s)\right)A_h^{\frac{\beta}{2}} P_hQ^{\frac{1}{2}}\Vert^2_{\mathcal{L}_2(H)}ds\right)^{1/2}\nonumber\\
&+&C\left(\sum_{j=0}^{m-1}\int_{t_j}^{t_{j+1}}
\left\Vert S_h(s)A_h^{\frac{\beta}{2}} P_hQ^{\frac{1}{2}}\right\Vert^2_{\mathcal{L}_2(H)}ds\right)^{1/2}\nonumber\\
&\leq& C\Delta t^{\frac{1}{2}-\epsilon}\left\Vert A^{\frac{\beta-1}{2}}Q^{\frac{1}{2}}\right\Vert_{\mathcal{L}_2(H)}\nonumber\\
&+&C\left(\int_0^{t_m}\left\Vert S_h(s)A_h^{\frac{1}{2}}\right\Vert^2_{L(H)}\left\Vert A_h^{\frac{\beta-1}{2}}P_h Q^{\frac{1}{2}}\right\Vert_{\mathcal{L}_2(H)}^2ds\right)^{1/2}\nonumber\\
&\leq& C\Delta t^{\frac{1}{2}-\epsilon}+C\left(\int_0^{T}\left\Vert S_h(s)A_h^{\frac{1}{2}}\right\Vert^2_{L(H)}ds\right)^{1/2}\leq C.
\end{eqnarray}
Since $\beta> \frac{d}{2}$ and the estimate \eqref{reaux1ab} is independent of $m$, using the Sobolev embedding \eqref{embedd0a}, it holds that
\begin{eqnarray}
\label{reaux1ac}
\sup_{M\in\mathbb{N}}\sup_{m\in\{0,\cdots, M\}}\mathbb{E}\left[\Vert W_{A_h}^m\Vert^{2p}_{\mathcal{C}(\overline{\Lambda}, \mathbb{R})}\right]\leq C.
\end{eqnarray}
Employing \rmref{remarkassumption3}, Cauchy-Schwartz's inequality, \eqref{reaux1ac} and \eqref{reaux1ab} it holds that
\begin{eqnarray}
\label{reaux1ad}
\Vert F(W_{A_h}^m)\Vert_{L^{2p}(\Omega, H)}\leq C\Vert W_{A_h}^m\Vert_{L^{4p}(\Omega, H)}\left(1+\mathbb{E}\Vert W_{A_h}^m\Vert^{4pc_1}_{\mathcal{C}(\overline{\Lambda}, \mathbb{R})}\right)\leq C,\; m\in\{0, \cdots, M\},\; M\in\mathbb{N}.
\end{eqnarray}
This completes the proof of \lemref{prioritylemma}.
\end{proof}

\begin{lemma}
\label{Momentboundlemma}
Under Assumptions \ref{assumption1}, \ref{assumption2}, \ref{assumption3} and \ref{assumption4},  the numerical scheme $X^h_m$ given in \eqref{scheme1}  satisfies   the following estimate 
\begin{eqnarray*}
\sup_{M\in\mathbb{N}}\sup_{m\in\{1,2,\cdots,M\}}\left(\Vert \xi^h_m\Vert_{L^{2p}(\Omega, H}\right)<\infty.
\end{eqnarray*}
\end{lemma}

\begin{proof}
We only give the proof for the backward Euler method, since the proof in the case of the semi-implicit method is similar.
We start by introducing the following process
\begin{eqnarray}
\label{reaux1}
Y^h_m:=X^h_m-W_{A_h}^m=S^m_{h,\Delta t}Y^h_0+\Delta t\sum_{j=0}^{m-1}S^{m-j}_{h,\Delta t}P_hF\left(Y^h_{j+1}+W_{A_h}^{j+1}\right),\; m\in\{1,2,\cdots, M\}.
\end{eqnarray}
One can easily check that $Y^h_m$ satisfies the following
\begin{eqnarray}
\label{reaux2}
\frac{Y^h_m-Y^h_{m-1}}{\Delta t}+A_hY^h_m=P_hF\left(Y^h_m+W^h_{A_h}\right),\quad Y^h_0=P_hX_0.
\end{eqnarray}
Taking the inner product in both sides of \eqref{reaux2},  using \assref{assumption2}  and Cauchy-Schwartz's inequality yields
\begin{eqnarray}
\label{reaux3}
\left\langle Y^h_m-Y^h_{m-1}, Y^h_m\right\rangle+\Delta t\langle A_hY^h_m, Y^h_m\rangle&=&\Delta t\left\langle F\left(Y^h_m+W^m_{A_h}\right), Y^h_m\right\rangle\nonumber\\
&=&\Delta t\left\langle F\left(Y^h_m+W^m_{A_h}\right)-F(W_{A_h}^m), Y^h_m\right\rangle+\Delta t\left\langle F(W_{A_h}^m), Y^h_m\right\rangle\nonumber\\
&\leq& L_1\Delta t\Vert Y^h_m\Vert^2+\frac{\Delta t}{2}\Vert F(W^m_{A_h})\Vert^2+\frac{\Delta t}{2}\Vert Y^h_m\Vert^2.
\end{eqnarray}
Using the fact that 
\begin{eqnarray}
\label{reaux4}
\frac{1}{2}\left(\Vert Y^h_m\Vert^2-\Vert Y^h_{m-1}\Vert^2\right)\leq \left\langle Y^h_m-Y^h_{m-1}, Y^h_m\right\rangle
\end{eqnarray}
it follows from \eqref{reaux3} that
\begin{eqnarray}
\label{reaux5}
\frac{1}{2}\left(\Vert Y^h_m\Vert^2-\Vert Y^h_{m-1}\Vert^2\right) +\Delta t\left\langle A_hY^h_m, Y^h_m\right\rangle\leq L_1\Delta t\Vert Y^h_m\Vert^2+\frac{\Delta t}{2}\Vert F(W^m_{A_h})\Vert^2+\frac{\Delta t}{2}\Vert Y^h_m\Vert^2.
\end{eqnarray}
From the coercivity estimate \eqref{ellip2}, \eqref{reaux5} and Cauchy-Schwartz's inequality, it holds that
\begin{eqnarray}
\label{reaux6}
\lambda_0\Delta t\Vert Y^h_m\Vert^2_1&\leq& \Delta ta\left(Y^h_m, Y^h_m\right)\nonumber\\
&\leq& \Delta t\langle A_hY^h_m, Y^h_m\rangle+\frac{1}{2}\left(\Vert Y^h_m\Vert^2-\Vert Y^h_{m-1}\Vert^2\right)-\frac{1}{2}\left(\Vert Y^h_m\Vert^2-\Vert Y^h_{m-1}\Vert^2\right)\nonumber\\
&\leq&\left(L_1+\frac{1}{2}\right)\Delta t\Vert Y^h_m\Vert^2+\frac{\Delta t}{2}\Vert F(W_{A_h})\Vert^2-\frac{1}{2}\left(\Vert Y^h_m\Vert^2-\Vert Y^h_{m-1}\Vert^2\right).
\end{eqnarray}
The equivalence of norms \cite[(2.12)]{Larsson2} and \cite[Lemma 3.1]{Antonio2} implies the existence of two positive constants $E_0$ and $E_1$ such that
\begin{eqnarray}
\label{reaux6a}
\Vert (A_h)^{1/2}A_h^{-1/2}\Vert_{L(H)}\leq E_0 \quad \text{and}\quad \Vert A_h^{1/2}u\Vert\leq E_1\Vert u\Vert_1,\quad u\in V_h.
\end{eqnarray}
Since  $Y^h_m\in V_h$, using Cauchy-Schwartz's inequality,  \eqref{reaux6a}  and \eqref{reaux6}, it holds that
\begin{eqnarray}
\label{reaux7}
\left\vert\left\langle A_hY^h_m, Y^h_m\right\rangle\right\vert&=&\left\vert\left\langle A_h^{1/2}Y^h_m, (A_h^*)^{1/2}Y^h_m\right\rangle\right\vert\leq \frac{1}{2}\Vert A_h^{1/2}Y^h_m\Vert^2+\frac{1}{2}\Vert (A_h^*)^{1/2}Y^h_m\Vert^2\nonumber\\
&\leq& \frac{E_1^2}{2}\Vert Y^h_m\Vert^2_1+\frac{1}{2}\Vert (A_h^*)^{1/2}A_h^{-1/2}\Vert^2_{L(H)}\Vert A_h^{1/2}Y^h_m\Vert^2\nonumber\\
&\leq& \frac{E_1^2}{2}\Vert Y^h_m\Vert^2_1+\frac{E_0^2}{2}\Vert A_h^{1/2}Y^h_m\Vert^2\leq (E_1^2+E_1^2E_0^2)\Vert Y^h_m\Vert^2_1\\
&\leq& \frac{(E_1^2+E_1^2E_0^2)}{\lambda_0\Delta t}\left[\left(L_1+\frac{1}{2}\right)\Delta t\Vert Y^h_m\Vert^2+\frac{\Delta t}{2}\Vert F(W_{A_h}^m)\Vert^2-\frac{1}{2}\left(\Vert Y^h_m\Vert^2-\Vert Y^h_{m-1}\Vert^2\right)\right].\nonumber
\end{eqnarray}
Substituting \eqref{reaux7} in \eqref{reaux5}  yields
\begin{eqnarray}
\label{reaux8}
\frac{1}{2}\left(\Vert Y^h_m\Vert^2-\Vert Y^h_{m-1}\Vert^2\right)
&\leq& L_1\Delta t\Vert Y^h_m\Vert^2+\frac{\Delta t}{2}\Vert F(W_{A_h}\Vert^2+\frac{\Delta t}{2}\Vert Y^h_m\Vert^2-\Delta t\left\langle A_hY^h_m, Y^h_m\right\rangle\nonumber\\
&\leq& \left(L_1+\frac{1}{2}\right)\Delta t\Vert Y^h_m\Vert^2+\frac{\Delta t}{2}\Vert F(W_{A_h}\Vert^2+\left\vert\Delta t\left\langle A_hY^h_m, Y^h_m\right\rangle\right\vert\nonumber\\
&\leq& \left(\frac{E_1^2+E_1^2E_0^2}{\lambda_0}+1\right)\left[\left(L_1+\frac{1}{2}\right)\Delta t\Vert Y^h_m\Vert^2+\frac{\Delta t}{2}\Vert F(W_{A_h}^m)\Vert^2\right]\nonumber\\
&-&\left(\frac{E_1^2+E_1^2E_0^2}{2\lambda_0}\right)\left(\Vert Y^h_m\Vert^2-\Vert Y^h_{m-1}\Vert^2\right).
\end{eqnarray}
It follows therefore  from \eqref{reaux8} and using \eqref{reaux1aa} that
\begin{eqnarray}
\label{reaux9}
&&\frac{1}{2}\left(1+\frac{E_1^2+E_1^2E_0^2}{\lambda_0}\right)\left(\Vert Y^h_m\Vert^2-\Vert Y^h_{m-1}\Vert^2\right)\nonumber\\
&\leq& \left(\frac{E_1^2+E_1^2E_0^2}{\lambda_0}+1\right)\left[\left(L_1+\frac{1}{2}\right)\Delta t\Vert Y^h_m\Vert^2+C\right].
\end{eqnarray}
Summing up \eqref{reaux9} over $m$ yields
\begin{eqnarray}
\label{reaux10}
\Vert Y^h_m\Vert^2&\leq &\Vert Y^h_0\Vert^2+2\Delta t\left(L_1+\frac{1}{2}\right)\left(\frac{\lambda_0}{E_1^2+E_1^2E_0^2+\lambda_0}\right)\sum_{j=1}^{m}\Vert Y^h_j\Vert^2+C.
\end{eqnarray}
 Applying the discrete Gronwall's lemma \cite[Lemma 7.1]{Larsson1}  to \eqref{reaux10}
yields
\begin{eqnarray}
\label{reaux12}
\Vert Y^h_m\Vert^2\leq C\left(\Vert X^h_0\Vert^2+1\right).
\end{eqnarray}
Taking the $L^{2p}(\Omega, H)$ norm in \eqref{reaux12} and using \assref{assumption1} yields
\begin{eqnarray*}
\Vert Y^h_m\Vert_{L^{2p}(\Omega, H)}\leq C+\Vert X_0\Vert_{L^{2p}(\Omega, H)}\leq C.
\end{eqnarray*}
The  proof of the lemma is completed  by using \eqref{reaux1aa}.
\end{proof}

\subsection{Main proof}
 Using triangle inequality we recast the fully discrete error as follows
\begin{eqnarray}
\label{main3}
\Vert X(t_m)-X^h_m\Vert_{L^{2p}(\Omega, H)}\leq \Vert X(t_m)-X^h(t_m)\Vert_{L^{2p}(\Omega, H)}+\Vert X^h(t_m)-X^h_m\Vert_{L^{2p}(\Omega, H)}.
\end{eqnarray}
 Note that the space error $\Vert X(t_m)-X^h(t_m)\Vert_{L^{2p}(\Omega, H)}$ is estimated in \thmref{Spaceerror}. It remains to estimate the time error $\Vert X^h(t_m)-X^h_m\Vert_{L^{2p}(\Omega, H)}$. 
 Let us  introduce the  auxiliary process $\widetilde{X}^h_m$, which satisfies
\begin{eqnarray}
\label{main1}
\widetilde{X}^h_m-\widetilde{X}^h_{m-1}+\Delta tA_h\widetilde{X}^h_m=\Delta tP_hF\left(X^h(t_m)\right)+P_h\Delta W_m,\quad \widetilde{X}^h_0=P_hX_0.
\end{eqnarray}
Note that the process $\widetilde{X}^h_m$ can be written as follows
\begin{eqnarray}
\label{main2}
\widetilde{X}^h_m=S^m_{h, \Delta t}P_hX_0+\Delta t\sum_{j=0}^{m-1}S^{m-j}_{h, \Delta t}P_hF\left(X^h(t_{j+1})\right)+W^m_{A_h}.
\end{eqnarray}
 Again using triangle inequality, we recast the time error as follows
\begin{eqnarray}
\label{main4}
\Vert X^h(t_m)-X^h_m\Vert_{L^{2p}(\Omega, H)}\leq\Vert X^h(t_m)-\widetilde{X}^h_m\Vert_{L^{2p}(\Omega, H)}+\Vert \widetilde{X}^h_m-X^h_m\Vert_{L^{2p}(\Omega, H)}.
\end{eqnarray}
Iterating the numerical scheme \eqref{scheme1} yields
\begin{eqnarray}
\label{aux12}
X^h_m=S^m_{h, \Delta t}P_hX_0+\Delta t\sum_{j=0}^{m-1}S^{m-j}_{h,\Delta t}P_hF(X^h_{j+1})+W_{A_h}^m.
\end{eqnarray}
We start with the estimate of the first term in \eqref{main4}.  Using triangle inequality, it holds that
{\small
\begin{eqnarray}
\label{main5}
\Vert X^h(t_m)-\widetilde{X}^h_m\Vert_{L^{2p}(\Omega, H)}
&\leq& \left\Vert\left(S_h(t_m)-S^m_{h,\Delta t}\right)P_hX_0\right\Vert_{L^{2p}(\Omega, H)}\nonumber\\
&+&\left\Vert\sum_{j=0}^{m-1}\int_{t_j}^{t_{j+1}}\left[S_h(t_m-s)P_hF(X^h(s))-S^{m-j}_{h, \Delta t}P_hF(X^h(t_{j+1}))\right]ds\right\Vert_{L^{2p}(\Omega, H)}\nonumber\\
&+&\left\Vert\sum_{j=0}^{m-1}\int_{t_j}^{t_{j+1}}\left(S_h(t_m-t_j)-S^{m-j}_{h, \Delta t}\right)P_hdW(s)\right\Vert_{L^{2p}(\Omega, H)}\nonumber\\
&=:& J_1+J_2+J_3.
\end{eqnarray}
}
Using \lemref{RecallMukam} (iv) and \assref{assumption1} yields
\begin{eqnarray}
\label{main5b}
J_1\leq C\Delta t^{\beta/2}.
\end{eqnarray}
In order to estimate $J_2$, we decompose it in three terms as follows:
\begin{eqnarray}
\label{main2a}
J_2&\leq&\left\Vert\sum_{j=0}^{m-1}\int_{t_j}^{t_{j+1}}S_h(t_m-s)P_h\left[F(X^h(s))-F(X^h(t_{j+1}))\right]ds\right\Vert_{L^{2p}(\Omega, H)}\nonumber\\
&+&\left\Vert\sum_{j=0}^{m-1}\int_{t_j}^{t_{j+1}}\left(S_h(t_m-s)-S_h(t_m-t_j)\right)P_hF(X^h(t_{j+1}))ds\right\Vert_{L^{2p}(\Omega, H)}\nonumber\\
&+&\left\Vert\sum_{j=0}^{m-1}\int_{t_j}^{t_{j+1}}\left(S_h(t_{m-j})-S^{m-j}_{h,\Delta t}\right)P_hF(X^h(t_{j+1}))ds\right\Vert_{L^{2p}(\Omega, H)}:=J_{21}+J_{22}+J_{23}.
\end{eqnarray}
We start by estimating $J_{22}$ and $J_{23}$ since they are easier than that of $J_{21}$. Using triangle inequality, the smoothing properties of the semigroup, \lemref{Sharpestimates} and \propref{propexistence} yields
\begin{eqnarray}
\label{main3a}
J_{22}&\leq& \sum_{j=0}^{m-1}\int_{t_j}^{t_{j+1}}\left\Vert\left(S_h(t_m-s)-S_h(t_m-t_j)\right)P_hF(X^h(t_{j+1}))\right\Vert_{L^{2p}(\Omega, H)}ds\nonumber\\
&\leq& \sum_{j=0}^{m-1}\int_{t_j}^{t_{j+1}}\left\Vert S_h(t_m-s)A_h^{\beta/2}\right\Vert_{L(H)}\left\Vert A_h^{-\beta/2}\left(\mathbf{I}-S_h(s-t_j)\right)\right\Vert_{L(H)}\Vert F(X^h(t_{j+1}))\Vert_{L^{2p}(\Omega, H)}ds\nonumber\\
&\leq& C\sum_{j=0}^{m-1}\int_{t_j}^{t_{j+1}}\left\Vert S_h(t_m-s)A_h^{\beta/2}\right\Vert_{L(H)}(s-t_j)^{\beta/2}ds\nonumber\\
&=&C\Delta t^{\beta/2}\sum_{j=0}^{m-1}\int_{t_j}^{t_{j+1}}\left\Vert S_h(t_m-s)A_h^{\beta/2}\right\Vert_{L(H)}ds\leq C\Delta t^{\beta/2}\int_0^{t_m}\left\Vert S_h(t_m-s)A_h^{\beta/2}\right\Vert_{L(H)}ds\nonumber\\
&\leq& C\Delta t^{\beta/2}.
\end{eqnarray}
Using triangle inequality, \lemref{RecallMukam} (v) and \propref{propexistence} yields
\begin{eqnarray}
\label{main4a}
J_{23}&\leq& \sum_{j=0}^{m-1}\int_{t_j}^{t_{j+1}}\left\Vert \left(S_h(t_{m-j})-S^{m-j}_{h,\Delta t}\right)P_hF(X^h(t_{j+1}))\right\Vert_{L^{2p}(\Omega, H)}ds\nonumber\\
&\leq& C\Delta t^{1-\epsilon}\sum_{j=0}^{m-1}\int_{t_j}^{t_{j+1}}t_{m-j}^{-1+\epsilon}\Vert P_hF(X^h(t_{j+1}))\Vert_{L^{2p}(\Omega, H)}ds\nonumber\\
&\leq& C\Delta t^{1-\epsilon}\sum_{j=0}^{m-1}\Delta tt_{m-j}^{-1+\epsilon}\leq C\Delta t^{1-\epsilon},
\end{eqnarray}
where at the last step we used the well-known estimate $\Delta t\sum\limits_{j=1}^mt_j^{-1+\alpha}<C$, for any $\alpha>0$.
Before proceeding to the estimate of $J_{21}$, we use Taylor's formula in Banach space to obtain 
\begin{eqnarray}
\label{main5a}
&&F(X^h(s))-F(X^h(t_{j+1}))\nonumber\\
&=&\left(\int_0^1F'\left(X^h(t_{j+1})+\lambda\left(X^h(s)-X^h(t_{j+1})\right)\right)d\lambda\right)\left(X^h(s)-X^h(t_{j+1})\right).
\end{eqnarray}
Note that the mild solution $X^h(s)$ can be written as follows
\begin{eqnarray}
\label{main6a}
X^h(t_{j+1})&=&S_h(t_{j+1}-s)X^h(s)+\int^{t_{j+1}}_sS_h(t_{j+1}-\sigma)P_hF(X^h(\sigma))d\sigma\nonumber\\
&+&\int^{t_{j+1}}_s
S_h(t_{j+1}-\sigma)P_hdW(\sigma),\quad s\in [t_j, t_{j+1}].
\end{eqnarray}
Substituting \eqref{main6a} in \eqref{main5a} yields
\begin{eqnarray}
\label{main7a}
F(X^h(s))-F(X^h(t_{j+1}))&=&-I^h_{j+1,s}\left(S_h(t_{j+1}-s)-\mathbf{I}\right)X^h(s)\nonumber\\
&-&I^h_{j+1,s}\int^{t_{j+1}}_sS_h(t_{j+1}-\sigma)P_hF(X^h(\sigma))d\sigma\nonumber\\
&-&I^h_{j+1,s}\int^{t_{j+1}}_sS_h(t_{j+1}-\sigma)P_hdW(\sigma), 
\end{eqnarray}
where 
\begin{eqnarray}
\label{main8a}
I^h_{k,s}:=\int_0^1F'\left(X^h(t_k)+\lambda\left(X^h(s)-X^h(t_k)\right)\right)d\lambda,\quad k\in\{0, 1,\cdots,M-1\}.
\end{eqnarray}
Note that using \rmref{remarkassumption3}, \thmref{Regularity} and \eqref{prior2} one can easily check that for any $p\geq 1$
\begin{eqnarray}
\label{main9a}
\Vert I^h_{k,s}\Vert_{L^{2p}(\Omega, H)}\leq C,\quad k\in\{0,\cdots,M-1\},\quad s\in[t_k, t_{k+1}]
\end{eqnarray}
Substituting \eqref{main7a} in the expression of $J_{21}$ from \eqref{main2a} and using triangle inequality yields
\begin{eqnarray}
\label{main10a}
J_{21}&\leq& \left\Vert\sum_{j=0}^{m-1}\int_{t_j}^{t_{j+1}}S_h(t_m-s) I^h_{j+1,s}\left(S_h(t_{j+1}-s)-\mathbf{I}\right)X^h(s)\right\Vert_{L^{2p}(\Omega, H)}\nonumber\\
&+&\left\Vert \sum_{j=0}^{m-1}\int_{t_j}^{t_{j+1}}S_h(t_m-s)I^h_{j+1,s}\int^{t_{j+1}}_sS_h(t_{j+1}-\sigma)P_hF(X^h(\sigma))d\sigma ds\right\Vert_{L^{2p}(\Omega, H)}\nonumber\\
&+&\left\Vert\sum_{j=0}^{m-1}\int_{t_j}^{t_{j+1}}S_h(t_m-s)I^h_{j+1,s}\int^{t_{j+1}}_sS_h(t_{j+1}-\sigma)P_hdW(\sigma)ds\right\Vert_{L^{2p}(\Omega, H)}\nonumber\\
&=:&J_{21}^{(1)}+J_{21}^{(2)}+J_{21}^{(3)}.
\end{eqnarray}
Using the smoothing properties of the semigroup, \eqref{main9a} and  \thmref{Regularity} yields
\begin{eqnarray}
\label{main11a}
J_{21}^{(1)}&\leq& \sum_{j=0}^{m-1}\int_{t_j}^{t_{j+1}}\Vert S_h(t_m-s)I^h_{j+1,s}\Vert_{L(H)}\left\Vert \left(S_h(t_{j+1}-s)-\mathbf{I}\right)A_h^{-\beta/2}A_h^{\beta/2}X^h(s)\right\Vert_{L^{2p}(\Omega, H)}ds\nonumber\\
&\leq& C\sum_{j=0}^{m-1}\int_{t_j}^{t_{j+1}}\left\Vert\left(S_h(t_{j+1}-s)-\mathbf{I}\right)A_h^{-\beta/2}\right\Vert_{L(H)}\left\Vert A_h^{\beta/2}X^h(s)\right\Vert_{L^{2p}(\Omega, H)}ds\nonumber\\
&\leq& C\sum_{j=0}^{m-1}\int_{t_j}^{t_{j+1}}(t_{j+1}-s)^{\beta/2}ds\leq C\Delta t^{\beta/2}.
\end{eqnarray}
Using triangle inequality, the smoothing properties of the semi-group, \eqref{main9a} and \thmref{Regularity} yields
\begin{eqnarray}
\label{main12a}
J_{21}^{(2)}&\leq& \sum_{j=0}^{m-1}\int_{t_j}^{t_{j+1}}\int^{t_{j+1}}_s\left\Vert S_h(t_m-s)I^h_{j+1,s}S_h(t_{j+1}-\sigma)P_hF(X^h(\sigma))\right\Vert_{L^{2p}(\Omega, H)}d\sigma ds\nonumber\\
&\leq& C\sum_{j=0}^{m-1}\int_{t_j}^{t_{j+1}}(t_{j+1}-s)ds\leq C\Delta t.
\end{eqnarray}
Using the stochastic Fubini's Theorem \cite{Prato,Prevot} and the Burkh\"{o}lder-Davis-Gundy's inequality yields
\begin{eqnarray}
\label{main13a}
J_{21}^{(3)}&=&\left\Vert\sum_{j=0}^{m-1}\int_{t_j}^{t_{j+1}}\int^{t_{j+1}}_{t_j}
\chi_{[s, t_j)}(\sigma)S_h(t_m-s)I^h_{j+1,s}S_h(t_{j+1}-\sigma)P_hdW(\sigma)ds\right\Vert_{L^{2p}(\Omega, H)}\\
&=&\left\Vert\sum_{j=0}^{m-1}\int_{t_j}^{t_{j+1}}\int^{t_{j+1}}_{t_j}
\chi_{[s, t_j)}(\sigma)S_h(t_m-s)I^h_{j+1,s}S_h(t_{j+1}-\sigma)P_hdsdW(\sigma)\right\Vert_{L^{2p}(\Omega, H)}\nonumber\\
&\leq& C\left(\sum_{j=0}^{m-1}\int_{t_j}^{t_{j+1}}\left\Vert \int_{t_j}^{t_{j+1}}\chi_{[s, t_j)}(\sigma)S_h(t_m-s)I^h_{j+1,s}S_h(t_{j+1}-\sigma)P_hQ^{\frac{1}{2}}ds\right\Vert^2_{L^{2p}(\Omega, \mathcal{L}_2(H))}d\sigma\right)^{\frac{1}{2}}.\nonumber
\end{eqnarray}
Using Cauchy-Schwartz's inequality,  \cite[Lemma 11]{Antjd1}, \eqref{main9a} and the smoothing properties of the semigroup yields
{\small
\begin{eqnarray}
\label{main14a}
J_{21}^{(3)}&\leq&C\Delta t^{1/2}\left(\sum_{j=0}^{m-1}\int_{t_j}^{t_{j+1}}\int_{t_j}^{t_{j+1}}\left\Vert
\chi_{[s, t_j)}(\sigma)S_h(t_m-s)I^h_{j+1,s}S_h(t_{j+1}-\sigma)P_hQ^{\frac{1}{2}}\right\Vert^2_{L^{2p}(\Omega, \mathcal{L}_2(H))}dsd\sigma\right)^{\frac{1}{2}}\nonumber\\
&\leq& C\Delta t^{1/2}\left(\sum_{j=0}^{m-1}\int_{t_j}^{t_{j+1}}\int_{t_j}^{t_{j+1}}\left\Vert S_h(t_{j+1}-\sigma)A_h^{\frac{1-\beta}{2}}\right\Vert^2_{L(H)}\left\Vert A_h^{\frac{\beta-1}{2}}P_hQ^{\frac{1}{2}}\right\Vert^2_{\mathcal{L}_2(H)}dsd\sigma\right)^{1/2}\nonumber\\
&\leq& C\Delta t^{1/2}\left(\sum_{j=0}^{m-1}\int_{t_j}^{t_{j+1}}\int_{t_j}^{t_{j+1}}(t_{j+1}-\sigma)^{\min(0, \beta-1)}d\sigma ds\right)^{\frac{1}{2}}\leq C\Delta t^{\min\left(1, \frac{1+\beta}{2}\right)}.
\end{eqnarray}
}
Substituting \eqref{main14a}, \eqref{main12a} and \eqref{main11a} in \eqref{main10a} yields
\begin{eqnarray}
\label{main15a}
J_{21}\leq C\Delta t^{\beta/2-\epsilon}.
\end{eqnarray}
Substituting \eqref{main15a}, \eqref{main4a} and \eqref{main3a} in \eqref{main2a} yields
\begin{eqnarray}
\label{main16a}
J_2\leq C\Delta t^{\beta/2-\epsilon}.
\end{eqnarray}
Using the Burkh\"{o}lder-Davis-Gundy's inequality \eqref{Davis1},  \eqref{main9a}, \cite[Lemma 11]{Antjd1} and \lemref{RecallMukam} (iii) yields
\begin{eqnarray}
\label{main17a}
J_3&\leq& \left(\sum_{j=0}^{m-1}\int_{t_j}^{t_{j+1}}\left\Vert \left(S_h(t_{m-j})-S^{m-j}_{h,\Delta t}\right)P_hQ^{\frac{1}{2}}\right\Vert^2_{L^{2p}(\Omega, \mathcal{L}_2(H))}ds\right)^{1/2}\nonumber\\
&\leq& C\left(\sum_{j=0}^{m-1}\int_{t_j}^{t_{j+1}}\Delta t^{\beta}t_{m-j}^{-1}\left\Vert A_h^{\frac{\beta-1}{2}}P_hQ^{\frac{1}{2}}\right\Vert^2_{\mathcal{L}_2(H)}ds\right)^{1/2}\nonumber\\
&\leq& C\Delta t^{\beta/2-\epsilon}\left(\sum_{j=0}^{m-1}\int_{t_j}^{t_{j+1}}t_{m-j}^{-1+\epsilon}ds\right)^{1/2}\leq C\Delta t^{\beta/2-\epsilon}.
\end{eqnarray}
Substituting \eqref{main17a}, \eqref{main16a} and \eqref{main5b} in \eqref{main5} yields
\begin{eqnarray}
\label{errortime1}
\Vert X(t_m)-\widetilde{X}^h(t_m)\Vert_{L^{2p}(\Omega, H)}\leq C\Delta t^{\beta/2-\epsilon}.
\end{eqnarray}
It remains to estimate $\Vert \widetilde{X}^h(t_m)-X^h_m\Vert_{L^{2p}(\Omega, H)}$. To do so, we introduce $\widetilde{e}^h_m:=\widetilde{X}^h(t_m)-X^h_m$. It is easy to see that $\widetilde{e}^h_m$ satisfies the following equation
\begin{eqnarray}
\label{numaux1}
\widetilde{e}^h_m-\widetilde{e}^h_{m-1}+\Delta tA_h\widetilde{e}^h_m=\Delta tP_h\left(F(X^h(t_m))-F(X^h_m)\right),\quad \widetilde{e}^h_0=0.
\end{eqnarray}
Taking the inner product in both sides of \eqref{numaux1} with $\widetilde{e}^h_m$ yields
\begin{eqnarray}
\label{numaux2}
\langle\widetilde{e}^h_m-\widetilde{e}^h_{m-1}, \widetilde{e}^h_m\rangle+\Delta t\langle A_h\widetilde{e}^h_m, \widetilde{e}^h_m\rangle=\Delta t\left\langle\left(F(X^h(t_m))-F(X^h_m)\right),\widetilde{e}^h_m\right\rangle.
\end{eqnarray}
Using \eqref{reaux4} and \assref{assumption2}, it follows that
\begin{eqnarray}
\label{numaux3}
&&\frac{1}{2}\left(\Vert\widetilde{e}^h_m\Vert^2-\Vert\widetilde{e}^h_{m-1}\Vert^2\right)+\Delta t\langle A_h\widetilde{e}^h_m, \widetilde{e}^h_m\rangle\nonumber\\
&\leq&\Delta t\left\langle\left(F(X^h(t_m))-F(X^h_m)\right),\widetilde{e}^h_m\right\rangle\nonumber\\
&=&\Delta t\left\langle\left(F(X^h(t_m))-F(\widetilde{X}^h(t_m))\right),\widetilde{e}^h_m\right\rangle+\Delta t\left\langle\left(F(\widetilde{X}^h(t_m))-F(X^h_m)\right),\widetilde{e}^h_m\right\rangle\nonumber\\
&\leq& \Delta t\left\langle\left(F(X^h(t_m))-F(\widetilde{X}^h(t_m))\right),\widetilde{e}^h_m\right\rangle+C\Delta t\Vert\widetilde{e}^h_m\Vert^2.
\end{eqnarray}
From the coercivity estimate \eqref{ellip2}, \eqref{numaux3} and Cauchy-Schwartz inequality, it holds that
\begin{eqnarray}
\label{numaux4}
\lambda_0\Delta t\Vert \widetilde{e}^h_m\Vert^2_1&\leq& \Delta ta\left(\widetilde{e}^h_m, \widetilde{e}^h_m\right)\nonumber\\
&\leq& \Delta t\langle A_h\widetilde{e}^h_m, \widetilde{e}^h_m\rangle+\frac{1}{2}\left(\Vert\widetilde{e}^h_m\Vert^2-\Vert\widetilde{e}^h_{m-1}\Vert^2\right)-\frac{1}{2}\left(\Vert\widetilde{e}^h_m\Vert^2-\Vert\widetilde{e}^h_{m-1}\Vert^2\right)\nonumber\\
&\leq& \Delta t\left\langle\left(F(X^h(t_m))-F(\widetilde{X}^h(t_m))\right),\widetilde{e}^h_m\right\rangle+C\Delta t\Vert\widetilde{e}^h_m\Vert^2-\frac{1}{2}\left(\Vert\widetilde{e}^h_m\Vert^2-\Vert\widetilde{e}^h_{m-1}\Vert^2\right)\nonumber\\
&\leq& \frac{\Delta t}{2}\Vert F(X^h(t_m))-F(\widetilde{X}^h(t_m))\Vert^2+\frac{\Delta t}{2}\Vert\widetilde{e}^h_m\Vert^2-\frac{1}{2}\left(\Vert\widetilde{e}^h_m\Vert^2-\Vert\widetilde{e}^h_{m-1}\Vert^2\right).
\end{eqnarray}
Using \eqref{numaux3}, Cauchy-Schwartz's inequality and \eqref{numaux4} yields
\begin{eqnarray}
\label{numaux5}
\frac{1}{2}\left(\Vert\widetilde{e}^h_m\Vert^2-\Vert\widetilde{e}^h_{m-1}\Vert^2\right)
&\leq&\Delta t\left\langle\left(F(X^h(t_m))-F(\widetilde{X}^h(t_m))\right),\widetilde{e}^h_m\right\rangle+C\Delta t\Vert\widetilde{e}^h_m\Vert^2-\Delta t\langle A_h\widetilde{e}^h_m, \widetilde{e}^h_m\rangle\nonumber\\
&\leq& \frac{\Delta t}{2}\Vert F(X^h(t_m))-F(\widetilde{X}^h(t_m))\Vert^2+\frac{\Delta t}{2}\Vert\widetilde{e}^h_m\Vert^2+C\Delta t\Vert\widetilde{e}^h_m\Vert^2\nonumber\\
&+&\Delta t\left\vert\left\langle A_h\widetilde{e}^h_m, \widetilde{e}^h_m\right\rangle\right\vert.
\end{eqnarray}
Since  $\widetilde{e}^h_m\in V_h$, using Cauchy-Schwartz's inequality,   the equivalence of norms \cite[(2.12)]{Larsson2} and \cite[Lemma 3.1]{Antonio2}, it holds that
\begin{eqnarray}
\label{numaux6}
\left\vert\left\langle A_h\widetilde{e}^h_m, \widetilde{e}^h_m\right\rangle\right\vert&=&\left\vert\left\langle A_h^{1/2}\widetilde{e}^h_m, (A_h^*)^{1/2}\widetilde{e}^h_m\right\rangle\right\vert\leq \frac{1}{2}\Vert A_h^{1/2}\widetilde{e}^h_m\Vert^2+\frac{1}{2}\Vert (A_h^*)^{1/2}\widetilde{e}^h_m\Vert^2\nonumber\\
&\leq& \frac{C}{2}\Vert \widetilde{e}^h_m\Vert^2_1+\frac{1}{2}\Vert (A_h^*)^{1/2}A_h^{-1/2}\Vert^2_{L(H)}\Vert A_h^{1/2}\widetilde{e}^h_m\Vert^2\nonumber\\
&\leq& \frac{C}{2}\Vert \widetilde{e}^h_m\Vert^2_1+\frac{C}{2}\Vert A_h^{1/2}\widetilde{e}^h_m\Vert^2\leq C\Vert \widetilde{e}^h_m\Vert^2_1.
\end{eqnarray}
Substituting \eqref{numaux6} in \eqref{numaux5} and using \eqref{numaux4} yields
\begin{eqnarray}
\label{numaux7}
\frac{1}{2}\left(\Vert\widetilde{e}^h_m\Vert^2-\Vert\widetilde{e}^h_{m-1}\Vert^2\right)
&\leq& \frac{\Delta t}{2}\Vert F(X^h(t_m))-F(\widetilde{X}^h(t_m))\Vert^2+\frac{\Delta t}{2}\Vert\widetilde{e}^h_m\Vert^2+C\Delta t\Vert\widetilde{e}^h_m\Vert^2\nonumber\\
&+&\frac{C}{\lambda_0}.\lambda_0\Delta t\Vert\widetilde{e}^h_m\Vert^2_1\nonumber\\
&\leq&C\Delta t\Vert F(X^h(t_m))-F(\widetilde{X}^h(t_m))\Vert^2+C\Delta t\Vert\widetilde{e}^h_m\Vert^2-\frac{C}{2\lambda_0}\left(\Vert\widetilde{e}^h_m\Vert^2-\Vert\widetilde{e}^h_{m-1}\Vert^2\right)\nonumber.
\end{eqnarray}
It follows from the above inequality that
\begin{eqnarray}
\label{numaux8}
\left(\frac{1}{2}+\frac{C}{2\lambda_0}\right)\left(\Vert\widetilde{e}^h_m\Vert^2-\Vert\widetilde{e}^h_{m-1}
\Vert^2\right)\leq C\Delta t\Vert F(X^h(t_m))-F(\widetilde{X}^h(t_m))\Vert^2+C\Delta t\Vert\widetilde{e}^h_m\Vert^2.
\end{eqnarray}
We can rewrite \eqref{numaux8} as follows
\begin{eqnarray}
\label{numaux9}
\Vert\widetilde{e}^h_m\Vert^2-\Vert\widetilde{e}^h_{m-1}
\Vert^2\leq C\Delta t\Vert F(X^h(t_m))-F(\widetilde{X}^h(t_m))\Vert^2+C\Delta t\Vert\widetilde{e}^h_m\Vert^2.
\end{eqnarray}
Summing over $m$ in both sides of \eqref{numaux9} yields
\begin{eqnarray}
\label{numaux10}
\Vert \widetilde{e}^h_m\Vert^2\leq C\sum_{j=0}^m\Vert \widetilde{e}^h_j\Vert^2+C\Delta t\sum_{j=0}^m\Vert F(X^h(t_j))-F(\widetilde{X}^h(t_j))\Vert^2.
\end{eqnarray}
Taking the $L^{p}$ norm in  \eqref{numaux10}, using \assref{assumption3}, \thmref{Regularity}, \eqref{errortime1} and \eqref{aux1b} yields
\begin{eqnarray}
\label{numaux11}
&&\Vert \widetilde{e}^h_m\Vert_{L^{2p}(\Omega, H)}\nonumber\\
&\leq& C\sum_{j=0}^m\Vert \widetilde{e}^h_j\Vert_{L^{2p}(\Omega,H)}+C\Delta t\sum_{j=0}^m\Vert F(X^h(t_j))-F(\widetilde{X}^h(t_j))\Vert_{L^{2p}(\Omega, H)}\nonumber\\
&\leq& C\sum_{j=0}^m\Vert \widetilde{e}^h_j\Vert_{L^{2p}(\Omega,H)}\nonumber\\
&+&C\Delta t\sum_{j=0}^m\Vert X^h(t_j)-\widetilde{X}^h(t_j)\Vert_{L^{4p}(\Omega, H)}\left(1+\Vert X^h(t_j)\Vert^{2c_1}_{L^{4pc_1}(\Omega, \mathcal{C}(\overline{\Lambda}, \mathbb{R}))}+\Vert\widetilde{X}^h(t_j)\Vert^{2c_1}_{L^{4pc_1}(\Omega, \mathcal{C}(\overline{\Lambda}, \mathbb{R}))}\right)\nonumber\\
&\leq&C\sum_{j=0}^m\Vert \widetilde{e}^h_j\Vert_{L^{2p}(\Omega,H)}+C\Delta t\sum_{j=0}^m\Delta t^{\beta/2-\epsilon}=C\Delta t^{\beta/2-\epsilon}+C\sum_{j=0}^m\Vert \widetilde{e}^h_j\Vert_{L^{2p}(\Omega,H)}.
\end{eqnarray}
Employing \lemref{Momentboundlemma} and \eqref{aux1b} it holds that $\Vert \widetilde{e}^h_m\Vert_{L^{2p}(\Omega, H)}\leq C$ for any $m\in\{0, \cdots, M\}$ and $M\in\mathbb{N}$. Therefore, applying the discrete Gronwall's lemma \cite[Lemma 7.1]{Larsson1} to \eqref{numaux11} yields
\begin{eqnarray}
\label{errortime2}
\Vert \widetilde{e}^h_m\Vert_{L^{2p}(\Omega, H)}\leq C\Delta t^{\beta/2-\epsilon}.
\end{eqnarray}
Substituting \eqref{errortime2}  and \eqref{errortime1} in \eqref{main4} completes the proof of \thmref{mainresult1}. 

 \section{Numerical simulations}
\label{numerik} 

We consider the following two dimensional  stochastic  reactive dominated advection
diffusion reaction  with  constant diagonal diffusion coefficient
\begin{eqnarray}
\label{reactiondif1}
dX=\left[ D \varDelta X-\nabla \cdot(\mathbf{q}X)- (X^5-X ) \right]dt+dW,
\end{eqnarray}
with  mixed Neumann-Dirichlet boundary conditions on $\Lambda=[0,L_1]\times[0,L_2]$. 
The Dirichlet boundary condition is $X=1$ at $\Gamma=\{ (x,y) :\; x =0\}$ and 
we use the homogeneous Neumann boundary conditions elsewhere.
Note $q$ is the Darcy velocity obtained as in \cite{Antonio1}.
Our noise has the  eigenfunctions $\{e_{i}^{(1)}e_{j}^{(2)}\}_{i,j\geq 0}
$  as the operator $-\varDelta$ with homogeneous  Neumann boundary conditions given by 
\begin{eqnarray}
e_{0}^{(l)}=\sqrt{\dfrac{1}{L_{l}}},\;\;\;\lambda_{0}^{(l)}=0,\;\;\;
e_{i}^{(l)}=\sqrt{\dfrac{2}{L_{l}}}\cos(\lambda_{i}^{(l)}x),\;\;\;\lambda_{i}^{(l)}=\dfrac{i
  \,\pi }{L_{l}},
\end{eqnarray}
where $l \in \left\lbrace 1, 2 \right\rbrace$ and  $i=\{1, 2, 3, \cdots\}$.
In the noise representation \eqref{noise}, we have used
\begin{eqnarray}
\label{noise2}
 q_{i,j}=\left( i^{2}+j^{2}\right)^{-(\beta +\delta)}, \, \beta>0,
\end{eqnarray} 
Here the noise and the linear operator are supposed to have the same
eigenfunctions. We obviously have 
\begin{eqnarray}
\underset{(i,j) \in \mathbb{N}^{2}}{\sum}\lambda_{i,j}^{\beta-1}q_{i,j}<  \pi^{2}\underset{(i,j) 
\in \mathbb{N}^{2}}{\sum} \left( i^{2}+j^{2}\right)^{-(1+\delta)} <\infty,
\end{eqnarray}
thus  \assref{assumption4} is fulfilled. Note that \assref{assumption3} is fulfilled with $\varphi(x)=-x^5+x$. It remains to prove that \assref{assumption2} is fulfilled. One can easily check that 
\begin{eqnarray}
\label{identite1}
\left\langle u-v, F(u)-F(v)\right\rangle&=&\Vert u-v\Vert^2-\left\langle u-v, u^5-v^5\right\rangle\nonumber\\
&=&\Vert u-v\Vert^2-\int_{\Lambda}\left(u(x)-v(x)\right)\left(u^5(x)-v^5(x)\right)dx,\quad u, v\in H.
\end{eqnarray}
Let us recall the following simple identity 
\begin{eqnarray}
\label{identite2}
a^5-b^5=(a-b)(a^4+a^3b+a^2b^2+ab^3+b^4),\quad a, b\in\mathbb{R}. 
\end{eqnarray}
We claim also that the following estimate holds 
\begin{eqnarray}
\label{identie3}
\psi(a, b):=a^4+a^3b+a^2b^2+ab^3+b^4\geq 0,\quad a, b\in \mathbb{R}.
\end{eqnarray}
In fact, we distinguish two situations:
\begin{itemize}
\item If $a\geq b$, then it follows that
\begin{eqnarray*}
\psi(a, b)=a^4+\underbrace{a^3b+a^2b^2}_{a^2b(a+b)}+\underbrace{ab^3+b^4}_{b^3(a+b)}\geq a^4+2a^2b^2+2b^4\geq 0.
\end{eqnarray*}
\item If $a\leq b$, then it follows that
\begin{eqnarray*}
\psi(a, b)=\underbrace{a^4+a^3b}_{a^3(a+b)}+\underbrace{a^2b^2+ab^3}_{ab^2(a+b)}+b^4\geq 2a^4+2a^2b^2+b^4\geq 0.
\end{eqnarray*}
\end{itemize}
This prove the claim. Using the identity \eqref{identite2} and the estimate \eqref{identie3}, it holds  that
\begin{eqnarray}
\label{identite4}
&&-\int_{\Lambda}\left(u(x)-v(x)\right)\left(u^5(x)-v^5(x)\right)dx\nonumber\\
&=&-\int_{\Lambda}\left(u(x)-v(x)\right)^2\left(u^4(x)+u^3(x)v(x)+u^2(x)v^2(x)+u(x)v^3(x)+v^4(x)\right)dx\leq 0.
\end{eqnarray}
Substituting \eqref{identite4} in \eqref{identite1} yields
\begin{eqnarray*}
\left\langle u-v, F(u)-F(v)\right\rangle\leq \Vert u-v\Vert^2.
\end{eqnarray*}
Hence \assref{assumption2} is fulfilled.  


In our simulations, we take $\beta =2$ and $\delta=0.001$.
%
\begin{figure}[!ht]
\begin{center}
 \includegraphics[width=0.5\textwidth]{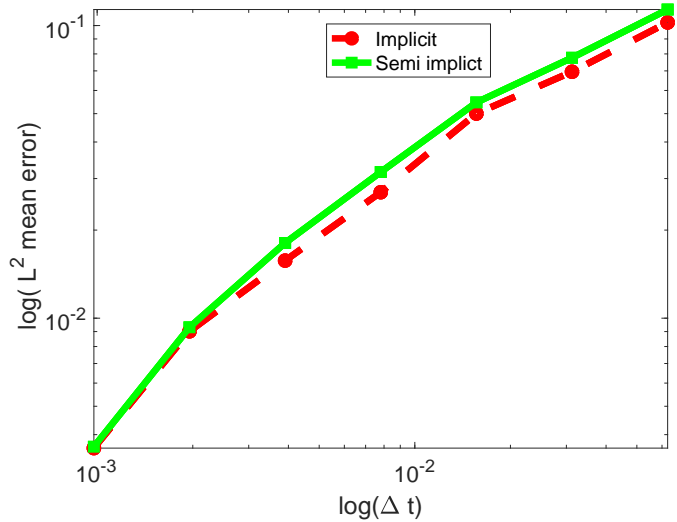}
  \end{center}
\caption{Convergence of the  implicit and semi implicit schemes with  $\beta=2$ and $\delta=0.001$ in \eqref{noise2} at the final time $T=1$.
The order of convergence in time  is $0.92$ and $0.93$ respectively. The total number of samples used is $50$. Note that the ''reference solution'' for each sample is the numerical solution with the smaller  time step
$\Delta t=1/2018$}
 \label{FIGII}
 \end{figure}
 
     
\end{document}